\documentclass[12pt]{amsart}

\usepackage{fullpage, graphicx, enumitem}
\usepackage{url}
\usepackage{verbatim}
\usepackage{bbm}
\usepackage{amssymb,amsmath}
\usepackage{amsfonts,savesym,graphicx,bm,amsthm,stmaryrd}
\usepackage{xcolor}
\usepackage[mathscr]{eucal}
\usepackage[all]{xy}
\usepackage{mathtools}
\usepackage{tikz-cd}

\definecolor{hot}{RGB}{65,105,225}
\usepackage[pagebackref=true,colorlinks=true, linkcolor=hot ,  citecolor=hot, urlcolor=hot]{hyperref}

\newtheorem{theorem}{Theorem}[section]
\newtheorem{lemma}[theorem]{Lemma}

\newtheorem{theorem-definition}[theorem]{Theorem-Definition}
\newtheorem{corollary}[theorem]{Corollary}
\newtheorem{proposition}[theorem]{Proposition}

\newtheorem{definition-theorem}[theorem]{Definition-Theorem}

\newtheorem{theorem-defintion}[theorem]{Theorem-Definition}
\newtheorem{corollary-definition}[theorem]{Corollary-Definition}
\newtheorem{definition-proposition}[theorem]{Definition-Proposition}

\theoremstyle{definition}

\newtheorem{definition}[theorem]{Definition}

\newtheorem{remark}[theorem]{Remark}

\numberwithin{equation}{section}
\def\bZ{\mathbb{Z}}

\def\bR{\mathbb{R}}
\def\bC{\mathbb{C}}

\def\bA{\mathbb{A}}
\def\A{\mathcal{A}}
\def\B{\mathcal{B}}
\def\C{\mathcal{C}}
\def\spb{\mathrm{Spb}}
\def\spr{\mathrm{Spr}}
\def\mspec{\mathrm{Max}}
\def\Cl{\mathrm{Cl}}
\def\gr{\mathrm{gr}}
\def\rad{\mathrm{rad}}
\def\Gal{\mathrm{Gal}}
\newcommand{\norm}[1]{\left\lVert#1\right\rVert}

\author{Tom Biesbrouck}
\address{Department of Mathematics, KU Leuven, Celestijnenlaan 200B, 3001 Leuven, Belgium.}
\email{tom.biesbrouck@kuleuven.be}
\title{An algebraic characterisation of non-Archimedean Stein spaces}

\begin{document}
\begin{abstract}
	We introduce Liu algebras as Banach algebras which are \lq locally affinoid\rq, and define non-Archimedean Stein algebras as suitable inverse limits of these. We show that this gives rise to a complete functorial characterisation of non-Archimedean Liu and Stein spaces as Berkovich spectra of their respective algebras, thereby resolving a conjecture of Michael Temkin. This can be interpreted as a non-Archimedean analytic version of Serre's criterion for affineness. Furthermore, we prove a criterion that distinguishes affinoid algebras within the category of Liu algebras, answering another conjecture of Temkin. We also prove a generalisation of the Gerritzen-Grauert Theorem for non-Archimedean Stein spaces.
 \end{abstract}

\maketitle
\thispagestyle{empty}
\vspace{-0.5cm}
{
	\hypersetup{linkcolor=black}
	\tableofcontents 
}

\section*{Introduction}

\subsection*{Motivation.}\label{subsMot} In classical algebraic geometry, Serre's criterion for affineness provides a full characterisation of affine schemes in terms of vanishing higher cohomology. Analogously, one might wonder whether the conventional \lq local building blocks\rq\ of non-Archimedean analytic geometry, i.e. Berkovich spectra of affinoid algebras, also admit such an equivalent description. However, by work of Qing Liu in the late eighties, this turned out to be impossible: in \cite{LiuParis}, Liu constructs a non-affinoid compact \(k\)-analytic space with affinoid normalisation; and in \cite[§5]{LiuTohoku}, he constructs a non-affinoid analytic domain of the two-dimensional unit disk, both with vanishing higher cohomology. Consequently, a compact non-Archimedean analytic space \(X\) with \lq vanishing higher cohomology\rq\ is nowadays referred to as a Liu space. More generally, when \(X\) is not necessarily compact, it is called a Stein space, inspired by the complex analytic setting.

One natural approach still to recover a non-Archimedean analytic version of Serre's criterion for affineness, would thus be to enlarge the category of Banach algebras of which we take Berkovich spectra to serve as \lq local building blocks\rq. In this paper, we achieve this by introducing Liu algebras as Banach algebras which are \lq locally affinoid\rq, and proving that Liu spaces correspond exactly to Berkovich spectra of these Liu algebras. Subsequently, as to also include general Stein spaces, we recall the definition of Berkovich spectra for Fr\'echet algebras and define Stein algebras as suitable inverse limits of Liu algebras. This enables us to characterise any Stein space as the Berkovich spectrum of a corresponding Stein algebra. A direct corollary of all this, is the initially desired non-Archimedean analytic version of Serre's criterion for affineness.

We remark that our algebraic definition of Liu algebras as \lq locally affinoid\rq\ Banach algebras was already conjectured to correspond to Liu spaces in the first prerprint version of \cite[§4]{TemkinPinch} on arXiv. A proof of this was then attempted in the first prerprint version of \cite[Appendix B]{Xia} on arXiv. However, this proof overlooked some essential subtleties and was therefore not published. Furthermore, a characterisation of dagger Stein spaces without boundary as Berkovich spectra of certain Fr\'echet algebras has already been worked out in \cite[§4.2]{BBBK}.

\subsection*{Background.}\label{subsBack} We recall the terminology and characterisation of Liu and Stein spaces as were introduced in \cite{MaculanPoineau}. We let \(k\) be a complete non-Archimedean field (possibly trivially valued) and consider non-Archimedean analytic spaces in the sense of Berkovich \cite[§1.2]{BerkovichIHES}. Whenever the word \lq strict\rq\ is used, \(k\) is assumed to be non-trivially valued.
\begin{definition}
	Let \(X\) be a \(k\)-analytic space.
	\begin{itemize}
		\item The \emph{holomorphic convex hull} of a compact subset \(K\) of \(X\) is defined as
		\[\widehat{K}_X\vcentcolon=\left\{x\in X\,\,\middle|\,\,|f(x)|\leq\norm{f}_K\text{ for all }f\in\mathcal{O}_X(X)\right\},\]
		where \(\norm{f}_K\vcentcolon=\sup_{x\in K}|f(x)|\) for all \(f\in\mathcal{O}_X(X)\).
		\item A coherent sheaf \(\mathcal{F}\) on \(X\) is called \emph{universally acyclic} if, for every analytic extension \(\ell/k\), \(\mathcal{F}_\ell\) is acyclic on \(X_\ell\), i.e. \(H^q(X_\ell,\mathcal{F}_\ell)=0\) for every \(q\in\bZ_{>0}\).
	\end{itemize}
	The \(k\)-analytic space \(X\) is said to be
	\begin{itemize}
		\item \emph{holomorphically separable} if for all distinct points \(x_1,x_2\in X\) there exists an \(f\in\mathcal{O}_X(X)\) such that \(|f(x_1)|\neq|f(x_2)|\);
		\item \emph{holomorphically convex} if for every compact subset \(K\) of \(X\), the holomorphic convex hull \(\widehat{K}_X\) of \(K\) is compact;
		\item \emph{cohomologically Stein} if every coherent sheaf \(\mathcal{F}\) of \(\mathcal{O}_X\)-modules on \(X\) is acyclic.
	\end{itemize}
\end{definition}
Note that when \(X\) is separated and compact, \(X\) is always holomorphically convex. Also, when \(k\) is non-trivially valued and \(X\) is separated and countable at infinity, then any acyclic coherent sheaf on \(X\) is universally acyclic by \cite[Theorem A.5]{MaculanPoineau}.
\begin{definition}
	A \emph{(strictly) \(k\)-Liu space} is a (strictly) \(k\)-analytic space \(X\) which is separated, holomorphically separable, compact and on which \(\mathcal{O}_X\) is universally acyclic.
\end{definition}
\begin{definition}
	A (strictly) \(k\)-analytic space \(X\) is said to be \emph{\(W\)-exhausted by (strictly) Liu domains} if it admits a cover \(\{D_i\}_{i\in\bZ_{\geq0}}\) for the \(G\)-topology by (strictly) \(k\)-Liu spaces such that for each \(i\in\bZ_{\geq0}\), \(D_i\) is contained in \(D_{i+1}\) and the restriction map \(\mathcal{O}_X(D_{i+1})\to\mathcal{O}_X(D_i)\) has dense image.
\end{definition}
\begin{definition}\label{def:stein}
	A \emph{(strictly) \(k\)-Stein space} is a separated (strictly) \(k\)-analytic space \(X\) countable at infinity, satisfying one of the following conditions, which are equivalent by \cite[Theorem 1.11]{MaculanPoineau}.
	\begin{enumerate}[label=(\roman*)]
		\item For every analytic extension \(\ell/k\), \(X_\ell\) is cohomologically Stein.
		\item \(X\) is holomorphically convex, holomorphically separable and \(\mathcal{O}_X\) is universally acyclic.
		\item \(X\) is \(W\)-exhausted by (strictly) Liu domains.
	\end{enumerate}
\end{definition}
\begin{remark}
	To be precise, it is not proven explicitly in \cite{MaculanPoineau} that a strictly \(k\)-Stein space admits a \(W\)-exhaustion by \emph{strictly} Liu domains. However, a careful analysis of the proofs of \cite[Lemma 5.3 and Claim 5.4]{MaculanPoineau} shows that one can indeed take the constructed \(W\)-exhaustion to be strict.
\end{remark}

\begin{definition}\label{def:ratdomstein}
	Let \(X\) be a (strictly) \(k\)-Stein space. A \emph{(strictly) rational domain} of \(X\) is a subset of \(X\) of the form
	\[X\left(r_\bullet^{-1}\frac{f_\bullet}{g}\right)\vcentcolon=\left\{x\in X\,\,\middle|\,\, |f_i(x)|\leq r_i|g(x)|\text{ for }i=1,\ldots,n\right\}\]
	where \(r_\bullet=(r_1,\ldots,r_n)\in\bR_{>0}^n\), \(g\in\mathcal{O}_X(X)\) and \(f_\bullet=(f_1,\ldots,f_n)\in\mathcal{O}_X(X)^n\) such that \(g,f_1,\ldots,f_n\) have no common zeroes. In the strict case, we set \(r_1=\ldots=r_n=1\).
\end{definition}

Lastly, we recall the notion of (graded) reduction from \cite[§3]{TemkinGraded}.
\begin{definition}\label{def:reduction}
	Let \(\A\) be a \(k\)-Banach algebra. The \emph{reduction} and \emph{graded reduction} of \(\A\) are respectively defined as
	\[\widetilde{\A}\vcentcolon=\frac{\left\{a\in\A\,\,\middle|\,\,\rho_\A(a)\leq 1\right\}}{\left\{a\in\A\,\,\middle|\,\,\rho_\A(a)< 1\right\}}\qquad\text{and}\qquad\widetilde{\A}_\gr\vcentcolon=\bigoplus_{r\in\bR_{>0}}\frac{\left\{a\in\A\,\,\middle|\,\,\rho_\A(a)\leq r\right\}}{\left\{a\in\A\,\,\middle|\,\,\rho_\A(a)< r\right\}},\]
	where \(\rho_\A:\A\to\bR_{\geq0}\) denotes the spectral radius function of \(\A\) (cf. Definition \ref{def:spb}).
\end{definition}

\subsection*{Statement of the results.}\label{subsMain} The main enterprise of this paper is to provide a complete functorial characterisation of \(k\)-Liu spaces (resp. \(k\)-Stein spaces) as Berkovich spectra of an explicit class of \(k\)-Banach algebras (resp. \(k\)-Fr\'echet algebras). To achieve this, we introduce (strictly) \(k\)-Liu algebras as \(k\)-Banach algebras that are \lq locally \(k\)-affinoid\rq, i.e. their Berkovich spectra admit \lq nice coverings by affinoid rational domains\rq. By construction, the proof of Tate's Acyclicity Theorem then carries over almost verbatim, endowing the Berkovich spectrum of a \(k\)-Liu algebra with a canonical acyclic \(k\)-analytic structure sheaf. More generally, we define a \(k\)-Stein algebra as a \(k\)-Fr\'echet algebra obtained as a suitable inverse limit of \(k\)-Liu algebras, and define the \(k\)-analytic structure on its Berkovich spectrum via gluing. For the definition of \(k\)-Fr\'echet algebras and their Berkovich spectra, see Section \ref{secFrech}. Ultimately, we obtain the following theorem.
\begin{theorem}[cf. Theorem \ref{thm:isocatliu} and Theorem \ref{thm:isocatstein}]\label{thm:isocatfull}
	The category of (strictly) \(k\)-Liu algebras (resp. \(k\)-Stein algebras) is anti-equivalent to the category of (strictly) \(k\)-Liu spaces (resp. \(k\)-Stein spaces). More precisely, taking Berkovich spectra of (strictly) \(k\)-Liu algebras (resp. \(k\)-Stein algebras) and taking global sections of structure sheaves on (strictly) \(k\)-Liu spaces (resp. \(k\)-Stein spaces) induce anti-isomorphisms of categories which are mutually inverse to each other.
\end{theorem}
\begin{remark}\label{rem:banfrechstr}
	In the theorem above, it is implicit that for a compact (resp. countable at infinity) \(k\)-analytic space \(X\), \(\mathcal{O}_X(X)\) can be endowed with a canonical \(k\)-Banach structure (resp. \(k\)-Fr\'echet structure). When \(X\) is compact, this is done by taking a finite affinoid covering \(X=\cup_iX_i\) and letting \(\norm{f}_X=\max_i\norm{f}_{X_i}\) for every \(f\in\mathcal{O}_X(X)\), where each \(\norm{\cdot}_{X_i}\) is a norm on \(\mathcal{O}_X(X_i)\). It is easily verified that different choices of \(\{X_i,\norm{\cdot}_{X_i}\}_i\) yield equivalent norms on \(\mathcal{O}_X(X)\). When \(X\) is countable at infinity, \(\mathcal{O}_X(X)\) can be endowed with a canonical \(k\)-Fr\'echet structure by taking an exhaustion \(\{D_i\}_{i\in\bZ_{\geq0}}\) of \(X\) by compact analytic domains, each of which induces a seminorm \(\norm{\cdot}_{D_i}\) on \(\mathcal{O}_X(X)\) via pullback. It is easily verified that different choices for \(\{D_i\}_{i\in\bZ_{\geq0}}\) yield equivalent families of seminorms \(\{\norm{\cdot}_{D_i}\}_{i\in\bZ_{\geq0}}\).
\end{remark}
As a direct corollary of Theorem \ref{thm:isocatfull}, we get the following non-Archimedean analytic version of Serre's criterion for affineness.
\begin{corollary}
	Let \(X\) be a separated \(k\)-analytic space. Then
	\begin{enumerate}[label=(\roman*)]
		\item \(X\) is compact (resp. countable at infinity), and
		\item for every analytic extension \(\ell/k\), \(X_\ell\) is cohomologically Stein
	\end{enumerate}
	if and only if \(\mathcal{O}_X(X)\) is a \(k\)-Liu algebra (resp. \(k\)-Stein algebra) and \(X\cong\spb(\mathcal{O}_X(X))\).
\end{corollary}
Moreover, we prove the following criterion to distinguish (strictly) \(k\)-affinoid algebras within the category of (strictly) \(k\)-Liu algebras. 
\begin{theorem}[cf. Theorem \ref{thm:affcrit}]
	Let \(k\) be stable and \(X\) a \(k\)-Liu space. Denote \(\A=\mathcal{O}_X(X)\).
	\begin{enumerate}[label=(\roman*)]
		\item \(X\) is \(k\)-affinoid if and only if \(\widetilde{\A}_\gr\) is a finitely generated \(\widetilde{k}_\gr\)-algebra and there exist elements \(f_1,\ldots,f_n\in\A\), with spectral radii \(r_1,\ldots,r_n\), such that the Laurent domains \(X(r_1^{-1}f_1^{-1}),\ldots,X(r_n^{-1}f_n^{-1})\) form an affinoid covering of \(X\).
		\item If \(k\) is non-trivially valued and \(X\) is strictly \(k\)-Liu, then \(X\) is strictly \(k\)-affinoid if and only if \(\widetilde{\A}\) is a finitely generated \(\widetilde{k}\)-algebra and there exist elements \(f_1,\ldots,f_n\), each with spectral radius \(1\), such that the strictly Laurent domains \(X(f_1^{-1}),\ldots,X(f_n^{-1})\) form an affinoid covering of \(X\).
	\end{enumerate}
\end{theorem}
\begin{remark}\label{rem:conjtemkin}
	The criterion above is a weakening of \cite[Conjecture 4.1.4]{TemkinPinch}, where Temkin did not require the existence of affinoid coverings by Laurent domains. However, a counterexample to Temkin's stronger conjecture can be found in \cite{LiuParis}, where Liu constructs a non-affinoid strictly \(k\)-Liu space \(X\) for which the reduction of \(\mathcal{O}_X(X)\) equals \(\widetilde{k}\).
\end{remark}

Lastly, after having slightly generalised the definition of Runge immersions, we will prove the following generalisation of the Gerritzen-Grauert Theorem.
\begin{theorem}[cf. Theorem \ref{thm:ggstein}]
	Let \(\varphi:Y\to X\) be a monomorphism of (strictly) \(k\)-Stein spaces with \(Y\) (strictly) \(k\)-affinoid. Then there exists an admissible covering \(\{R_i\}_{i\in I}\) of \(X\) by (strictly) rational domains such that \(\varphi\) induces (strictly) Runge immersions \(\varphi_i:\varphi^{-1}(R_i)\to R_i\).
\end{theorem}
It follows that every (strictly) Liu domain in a (strictly) \(k\)-Stein space \(X\) can be written as a finite union of (strictly) \(k\)-affinoid connected components of (strictly) rational domains (cf. Corollary \ref{cor:affinoidrational}).

\subsection*{Structure of the paper.}\label{subsStruct} This paper consists of two parts. The first part, ranging from Section \ref{secRat} to Section \ref{secAffcrit}, is concerned solely with \(k\)-Liu algebras and \(k\)-Liu spaces. We start by recalling the notions of rational localisations and rational domains (Section \ref{secRat}), after which we introduce \(k\)-Liu algebras (Section \ref{secLiuAlg}). Next, we define Liu domains for Berkovich spectra of \(k\)-Liu algebras, and demonstrate that rational domains are Liu domains (Section \ref{secLiuDom}). We then prove a version of Tate's Acyclicity Theorem for \(k\)-Liu algebras (Section \ref{secTate}) and derive the correspondence between \(k\)-Liu algebras and \(k\)-Liu spaces (Section \ref{secLiuSpaces}). We conclude with the criterion for distinguishing \(k\)-affinoid algebras within the category of \(k\)-Liu algebras (Section \ref{secAffcrit}).

In the second part, ranging from Section \ref{secFrech} to Section \ref{secGG}, we consider general \(k\)-Stein algebras and \(k\)-Stein spaces. First, we recall the definition of a \(k\)-Fr\'echet algebra and its Berkovich spectrum (Section \ref{secFrech}), after which we define \(k\)-Stein algebras and prove their correspondence with \(k\)-Stein spaces (Section \ref{secStein}). Lastly, we slightly extend the notion of Runge immersions (Section \ref{secRunge}), and use this to prove a generalisation of the Gerritzen-Grauert Theorem to \(k\)-Stein spaces (Section \ref{secGG}).

\subsection*{Conventions.}\label{subsNot} Throughout this paper, we let \(k\) be a complete non-Archimedean field (possibly trivially valued) and consider non-Archimedean analytic spaces in the sense of Berkovich \cite[§1.2]{BerkovichIHES}. Whenever the word \lq strict\rq\ is used, \(k\) is assumed to be non-trivially valued. An analytic extension of \(k\) is a field extension \(\ell/k\) where \(\ell\) is also a complete non-Archimedean field of which the absolute value restricts to that of \(k\). For a polyradius \(r=(r_1,\ldots,r_n)\in\bR_{>0}^n\), we denote by \(k_r\) the completed fraction field of \(k(T_1,\ldots,T_n)\) with the generalised Gauss norm given by \(\norm{T_i}=r_i\) for each \(i\in\{1,\ldots,n\}\). An analytic extension \(\ell/k\) is called topgebraic if it can be embedded into \(\widehat{k^a}/k\), with \(k^a\) an algebraic closure of \(k\).

\subsection*{Acknowledgements.}\label{subsAck} The author was supported by the grant 11P4V24N from FWO. I am thankful to Johannes Nicaise for introducing me to a project on non-Archimedean Stein spaces, of which this paper became a side project. I also want to thank Mingchen Xia and Michael Temkin for kindly answering my emails when I stumbled upon this project. Lastly, many thanks go to J\'er\^ome Poineau for showing interest in this work, providing feedback to a first draft and pointing out some essential subtleties that were overlooked in Mingchen's first attempt at an algebraic characterisation of Liu spaces.

\section{Rational domains and rational localisations}\label{secRat}
In this section, we recall basic notions concerning Berkovich spectra of \(k\)-Banach algebras. In particular, we focus on rational domains and their connection with rational localisations of \(k\)-Banach algebras.
\begin{definition}\label{def:spb}
	Let \(\A\) be a \(k\)-Banach algebra.
	\begin{enumerate}[label=(\roman*)]
		\item The \emph{Berkovich spectrum of \(\A\)} is the set
		\[\spb(\A)\vcentcolon=\left\{\text{bounded multiplicative seminorms }|\cdot(x)|:\A\to\bR_{\geq0}\right\}\]
		equipped with the weakest topology such that \(\spb(\A)\to\bR_{\geq0}:|\cdot(x)|\mapsto|f(x)|\) is continuous for every \(f\in\A\).
		\item The \emph{spectral radius function of \(\A\)} is given by
		\[\rho_\A:\A\to\bR_{\geq0}:f\mapsto\inf_{n\in\bZ_{\geq0}}\sqrt[n]{\norm{f^n}}=\max_{x\in\spb(\A)}|f(x)|.\]
		\item \(\A\) is called \emph{uniform} if its norm \(\norm{\cdot}_\A\) is power-multiplicative.
		\item \(\A\) is called a \emph{\(k\)-Banach function algebra} if its spectral radius function \(\rho_\A\) is a complete power-multiplicative norm.
	\end{enumerate}
\end{definition}
\begin{remark}
	Recall that taking Berkovich spectra is contravariantly functorial, i.e. a bounded morphism \(\varphi:\A\to\B\) of \(k\)-Banach algebras induces a continuous map \(\spb(\varphi):\spb(\B)\to\spb(\A)\) via pullback. For the equality in the definition of \(\rho_\A\), see \cite[Theorem 1.3.1]{Berkovich}. When \(k\) is non-trivially valued, it follows from Banach's Open Mapping Theorem \cite[Th\'eor\`eme I.3.3/1]{BourbakiVec} that \(\A\) is a \(k\)-Banach function algebra if and only if \(\norm{\cdot}_\A\) and \(\rho_\A\) are equivalent as seminorms. Note that a uniform \(k\)-Banach algebra is always a \(k\)-Banach function algebra.
\end{remark}
\begin{definition}\label{def:rationaldomain}
	Let \(\A\) be a \(k\)-Banach algebra and denote \(X=\spb(\A)\).
	\begin{enumerate}[label=(\roman*)]
		\item A \emph{rational domain of \(X\)} is a subset of the form \[X\left(r_\bullet^{-1}\frac{f_\bullet}{g}\right)\vcentcolon=\left\{x\in X\,\,\middle|\,\,|f_i(x)|\leq r_i|g(x)|\text{ for all }i\right\},\]
		where \(r_\bullet=(r_1,\ldots,r_n)\in\bR_{>0}^n\), \(g\in\A\) and \(f_\bullet=(f_1,\ldots,f_n)\in\A^n\) such that \(g,f_1,\ldots,f_n\) generate the unit ideal in \(\A\).
		\item A \emph{Laurent domain of \(X\)} is a subset of the form
		\[X\left(r_\bullet^{-1}f_\bullet,s_\bullet g_\bullet^{-1}\right)\vcentcolon=\left\{x\in X\,\,\middle|\,\,|f_i(x)|\leq r_i\text{ and }|g_j(x)|\geq s_j\text{ for all }i,j\right\},\]
		where \(r_\bullet=(r_1,\ldots,r_n)\in\bR_{>0}^n\), \(s_\bullet=(s_1,\ldots,s_m)\in\bR_{>0}^m\), \(f_\bullet=(f_1,\ldots,f_n)\in\A^n\) and \(g_\bullet=(g_1,\ldots,g_m)\in\A^m\).
		\item A \emph{Weierstrass domain of \(X\)} is a Laurent domain for which we can set \(m=0\) in the notation above.
	\end{enumerate}
	When \(k\) is non-trivially valued, we define \emph{strictly rational} (resp. \emph{Laurent}, \emph{Weierstrass}) domains as rational (resp. Laurent, Weierstrass) domains for which we can take each \(r_i=1\) (and \(s_j=1\)).
\end{definition}
\begin{definition}\label{def:ratlocalisation}
	Let \(\A\) be a \(k\)-Banach algebra. A \emph{rational localisation} of \(\A\) is a \(k\)-Banach algebra of the form
	\[\A\left\langle r_\bullet^{-1}\frac{f_\bullet}{g}\right\rangle\vcentcolon=\frac{\A\langle r_1^{-1}T_1,\ldots,r_n^{-1}T_n\rangle}{\Cl(gT_1-f_1,\ldots,gT_n-f_n)}\]
	equipped with its residue norm, where \(r_\bullet=(r_1,\ldots,r_n)\in\bR_{>0}^n\), \(g\in\A\) and \(f_\bullet=(f_1,\ldots,f_n)\in\A^n\) such that \(g,f_1,\ldots,f_n\) generate the unit ideal in \(\A\). Here, \(\Cl(gT_1-f_1,\ldots,gT_n-f_n)\) denotes the closure of the ideal \((gT_1-f_1,\ldots,gT_n-f_n)\) in \(\A\). When \(k\) is non-trivially valued, \(\A\langle r_\bullet^{-1}f_\bullet/g\rangle\) is called a \emph{strictly rational localisation} if we can take each \(r_i=1\). Analogously, one defines \emph{(strictly) Laurent and Weierstrass localisations}.
\end{definition}
The following proposition establishes the expected relation between rational domains and rational localisations.
\begin{proposition}(\cite[Proposition 2.12]{Mihara})\label{prop:ratdesc}
	Let \(\A\) be a \(k\)-Banach algebra and denote \(X=\spb(\A)\). For a rational localisation \(\A\langle r_\bullet^{-1}f_\bullet/g\rangle\) of \(\A\), the canonical morphism \(\A\to\A\langle r_\bullet^{-1}f_\bullet/g\rangle\) induces a homeomorphism of \(\spb(\A\langle r_\bullet^{-1}f_\bullet/g\rangle)\) onto \(X(r_\bullet^{-1}f_\bullet/g)\).
\end{proposition}
\begin{remark}
	To be precise, in \cite{Mihara}, it is assumed that \(k\) is non-trivially valued, and rational domains and rational localisations are defined slightly differently, i.e. there one assumes that \(f_1,\ldots,f_n\) generates the unit ideal instead of \(g,f_1,\ldots,f_n\). However, one easily verifies that the proof of the cited proposition above remains valid in our setting.
\end{remark}
\begin{remark}
	Recall that a bounded morphism \(\varphi:\A\to\B\) of seminormed groups is called \emph{admissible} when the restriction of the norm on \(\B\) to \(\varphi(\A)\) is equivalent to the residue seminorm on \(\varphi(\A)\cong\A/\ker(\varphi)\). Note that the composition of admissible morphisms is again admissible, and that the image of an admissible \(k\)-linear map between \(k\)-Banach algebras is always closed.
\end{remark}
The following proposition is a generalisation of \cite[Propositions 2.17 and 2.18]{Mihara}.
\begin{proposition}\label{prop:closedrationalidealsgeneral}
	Let \(\A\) be a \(k\)-Banach function algebra. For any \(g,f_1,\ldots,f_n\in\A\) generating the unit ideal and \(r_1,\ldots,r_n\in\bR_{>0}\), the ideal \(I=(gT_1-f_1,\ldots,gT_n-f_n)\) is closed in \(\A\langle r_1^{-1}T_1,\ldots,r_n^{-1}T_n\rangle\).
\end{proposition}
\begin{proof}
	First note that for \(r=(r_1,\ldots,r_n)\), we get that \(I=\A\langle r_\bullet^{-1}T_\bullet\rangle\cap I(\A\langle r_\bullet^{-1}T_\bullet\rangle\widehat{\otimes}_kk_r)\) in \(\A\langle r_\bullet^{-1}T_\bullet\rangle\widehat{\otimes}_kk_r\). Since \(\A\langle r_\bullet^{-1}T_\bullet\rangle\) is closed in \(\A\langle r_\bullet^{-1}T_\bullet\rangle\widehat{\otimes}_kk_r\), it thus suffices to prove that \(I(\A\langle r_\bullet^{-1}T_\bullet\rangle\widehat{\otimes}_kk_r)\) is closed. Hence, we may assume that \(k\) is non-trivially valued and \(r_1=\ldots=r_n=1\). Moreover, as \(\A\) is a \(k\)-Banach function algebra, we may also assume that the norm on \(\A\) is given by the spectral radius function \(\rho_\A\). Let us now write \(\B=\A\langle T_1,\ldots,T_n\rangle\) and denote for every \(x\in\spb(\A)\) by \(\widetilde{x}\in\spb(\B)\) the seminorm defined by
	\[|G(\widetilde{x})|\vcentcolon=\sup_{I\in\bZ_{\geq0}^n}|G_I(x)|\quad\text{ for each }\quad G=\sum_{I\in\bZ_{\geq0}^n}G_IT^I\in\B.\]
	Note that \(\widetilde{x}\) is maximal among the preimage of \(x\) under \(\spb(\B)\to\spb(\A)\). As \(\spb(\A)\) is compact and \(g,f_1,\ldots,f_n\) generate the unit ideal in \(\A\), there is an \(\varepsilon\in\bR_{>0}\) such that
	\[\inf_{x\in\spb(\A)}\max\left\{|f_1(x)|,\ldots,|f_n(x)|,|g(x)|\right\}\geq\varepsilon.\]
	As a result, we have for each \(x\in\spb(\A)\) that
	\begin{equation}\label{eq:lowerbound}
		\max_{1\leq j\leq n}|(gT_j-f_j)(\widetilde{x})|=\max\left\{|f_1(x)|,\ldots,|f_n(x)|,|g(x)|\right\}\geq\varepsilon.
	\end{equation}
	Now define the linear maps
	\begin{align*}
		\varphi_1&:\B^n\to\B^{n^2}:(b_i)_{1\leq i\leq n}\mapsto(b_i(gT_j-f_j))_{1\leq i,j\leq n} \\
		\varphi_2&:\B^{n^2}\to\B:(b_{ij})_{1\leq i,j\leq n}\mapsto\sum_{1\leq i\leq n}b_{ii}.
	\end{align*}
	Note that their composition \(\varphi_2\circ\varphi_1\) is given by
	\[\varphi:\B^n\to\B:(b_1,\ldots,b_n)\mapsto b_1(gT_1-f_1)+\ldots+b_n(gT_n-f_n),\]
	so \(\varphi(\B^n)=(gT_1-f_1,\ldots,gT_n-f_n)\). Hence, to show that the latter ideal is closed in \(\B\), it suffices to show that both \(\varphi_1\) and \(\varphi_2\) are admissible. First take \(b=(b_i)_{1\leq i\leq n}\in\B^n\). By using the assumption that \(\norm{\cdot}_\B=\rho_\B\) and applying equation (\ref{eq:lowerbound}), we get that
	\begin{align*}
		\norm{b}_{\B^n}&=\max_{1\leq i\leq n}\norm{b_i}_\B \\
		&=\max_{1\leq i\leq n}\max_{x\in\spb(\A)}|b_i(\widetilde{x})| \\
		&\leq\frac{1}{\varepsilon}\max_{1\leq i\leq n}\max_{x\in\spb(\A)}\max_{1\leq j\leq n}|(gT_j-f_j)(\widetilde{x})|\cdot|b_i(\widetilde{x})| \\
		&=\frac{1}{\varepsilon}\max_{1\leq i\leq n}\max_{x\in\spb(\A)}\max_{1\leq j\leq n}|(gT_j-f_j)b_i(\widetilde{x})| \\
		&=\frac{1}{\varepsilon}\max_{1\leq i\leq n}\max_{1\leq j\leq n}\norm{(gT_j-f_j)b_i}_\B \\
		&=\frac{1}{\varepsilon}\norm{\varphi_1(b)}_{\B^{n^2}} \\
		&\leq\frac{1}{\varepsilon}\max_{1\leq j\leq n}\norm{gT_j-f_j}_\B\max_{1\leq i\leq n}\norm{b_i}_{\B} \\
		&=\frac{1}{\varepsilon}\max_{1\leq j\leq n}\norm{gT_j-f_j}_\B\norm{b}_{\B^n},
	\end{align*}
	or summarised
	\[\varepsilon\norm{b}_{\B^n}\leq\norm{\varphi_1(b)}_{\B^{n^2}}\leq\left(\max_{1\leq j\leq n}\norm{gT_j-f_j}_\B\right)\norm{b}_{\B^n},\]
	showing that \(\varphi_1\) is admissible. Next, take \(b'=(b_{ij}')_{1\leq i,j\leq n}\in\B^{n^2}\). Then we have that
	\begin{equation*}
		\norm{\varphi_2(b')}_\B=\norm{\sum_{1\leq i\leq n}b_{ii}'}_\B \leq\max_{1\leq i\leq n}\norm{b_{ii}'}_\B \leq\max_{1\leq i\leq n}\max_{1\leq j\leq n}\norm{b_{ij}'}_\B =\norm{b'}_{\B^{n^2}},
	\end{equation*}
	showing that \(\varphi_2\) is a bounded, and thus continuous since \(k\) was assumed to be non-trivially valued. As \(\varphi_2\) is also surjective, it follows from Banach's Open Mapping Theorem that it is also admissible.
\end{proof}
The following lemma shows that rational localisation commutes with base change. Although the proof is straightforward, we state it here for convenience of later reference.
\begin{lemma}\label{lem:ratcombas}
	Let \(\A\) be a \(k\)-Banach algebra and \(\ell/k\) an analytic extension. Then \(\A\langle r_\bullet^{-1}f_\bullet/g\rangle\widehat{\otimes}_k\ell\cong(\A\widehat{\otimes}_k\ell)\langle r_\bullet^{-1}f_\bullet/g\rangle\).
\end{lemma}
\begin{proof}
	Let us denote \(I=(gT_1-f_1,\ldots,gT_n-f_n)\subset\A\langle r_\bullet^{-1}T_\bullet\rangle\). By definition of rational localisation, we have an exact and admissible sequence 
	\[0 \to \Cl(I) \to \A\langle r_\bullet^{-1}T_\bullet\rangle \to \A\left\langle r_\bullet^{-1}\frac{f_\bullet}{g}\right\rangle \to 0.\]
	Since \(-\widehat{\otimes}_k\ell\) is exact and preserves admissibility, the induced sequence
	\[0 \to \Cl(I)\widehat{\otimes}_k\ell \to \A\langle r_\bullet^{-1}T_\bullet\rangle\widehat{\otimes}_k\ell \to \A\left\langle r_\bullet^{-1}\frac{f_\bullet}{g}\right\rangle\widehat{\otimes}_k\ell \to 0\]
	is also exact and admissible. One easily verifies that
	\[\Cl(I)\widehat{\otimes}_k\ell=\Cl\left(I(\A\langle r_\bullet^{-1}T_\bullet\rangle\widehat{\otimes}_k\ell)\right) \qquad\text{and}\qquad \A\langle r_\bullet^{-1}T_\bullet\rangle\widehat{\otimes}_k\ell\cong(\A\widehat{\otimes}_k\ell)\langle r_\bullet^{-1}T_\bullet\rangle.\]
	Hence, it follows by definition of rational localisations that \(\A\langle r_\bullet^{-1}f_\bullet/g\rangle\widehat{\otimes}_k\ell\cong(\A\widehat{\otimes}_k\ell)\langle r_\bullet^{-1}f_\bullet/g\rangle\).
\end{proof}
\section{Liu algebras}\label{secLiuAlg}
In this section, we introduce \(k\)-Liu algebras as \(k\)-Banach algebras which \lq locally look like \(k\)-affinoid algebras\rq. Although this name only makes sense a posteriori when having proven their correspondence with \(k\)-Liu spaces (cf. Theorem \ref{thm:isocatliu}), we already stick with it from the beginning to avoid notational cluttering. Our definition of \(k\)-Liu algebras could already be found in \cite[§4]{TemkinPinch} and Appendix B of the first version of \cite{Xia} on arXiv, where they were called \emph{locally \(k\)-affinoid algebras}.

\begin{definition}
	A morphism \(\varphi:\A\to\B\) of \(k\)-Banach algebras is called \emph{(strictly) rationally injective} if for each (strictly) rational localisation \(\A\langle r_\bullet^{-1}f_\bullet/g\rangle\) of \(\A\) the induced base change morphism \(\varphi\otimes\A\langle r_\bullet^{-1}f_\bullet/g\rangle:\A\langle r_\bullet^{-1}f_\bullet/g\rangle\to\B\widehat{\otimes}_\A\A\langle r_\bullet^{-1}f_\bullet/g\rangle\) is injective.
\end{definition}
\begin{definition}\label{def:liualg}
	A \(k\)-Banach algebra \(\A\) is called a \emph{(strictly) \(k\)-Liu algebra} if there exist finitely many (strictly) rational localisations \(\A_1,\ldots,\A_n\) of \(\A\), such that
	\begin{enumerate}[label=(\roman*)]
		\item each \(\A_i\) is (strictly) \(k\)-affinoid;
		\item the (strictly) rational domains \(\spb(\A_i)\) cover \(\spb(\A)\);
		\item the canonical map \(\iota:\A\to\prod_i\A_i\) is (strictly) rationally injective.
	\end{enumerate}
	Such a collection of (strictly) rational localisations of \(\A\) will be called a \emph{(strictly) rational atlas} of \(\A\). We define the \emph{category of (strictly) \(k\)-Liu algebras} as the category consisting of (strictly) \(k\)-Liu algebras with bounded \(k\)-Banach algebra morphisms between them.
\end{definition}
\begin{remark}\label{rem:faithflat}
	Note that a (strictly) rational localisation of a (strictly) \(k\)-Liu algebra is again a (strictly) \(k\)-Liu algebra. Also, the canonical map \(\iota:\A\to\prod_i\A_i\) is admissible by Banach's Open Mapping Theorem when \(k\) is non-trivially valued, and thus in general as well by \cite[Proposition 2.1.2(ii)]{Berkovich}. Moreover, one obtains an equivalent definition for \(k\)-Liu algebras when replacing conditions (ii)-(iii) with the requirement for \(\iota:\A\to\prod_i\A_i\) to be faithfully flat (cf. Corollary \ref{cor:faithful}).
\end{remark}
\begin{proposition}\label{prop:closedrationalidealsliu}
	Let \(\A\) be a \(k\)-Liu algebra. For any \(g,f_1,\ldots,f_n\in\A\) generating the unit ideal and \(r_1,\ldots,r_n\in\bR_{>0}\), the ideal \((gT_1-f_1,\ldots,gT_n-f_n)\) is closed in \(\A\langle r_1^{-1}T_1,\ldots,r_n^{-1}T_n\rangle\).
\end{proposition}
\begin{proof}
	Let \(\A_1,\ldots,\A_n\) be a rational atlas of \(\A\). By definition, we have a Cartesian diagram of \(k\)-Banach algebras
	\[\begin{tikzcd}
		\mathcal{B}\vcentcolon=\A\times_{\prod_i\A_i}\prod_i\mathcal{T}_i \arrow[d, "\varphi'", two heads] \arrow[r, "\iota'", hook] & \prod_i\mathcal{T}_i \arrow[d, "\varphi", two heads] \\
		\A \arrow[r, "\iota", hook]                                                                            & \prod_i\A_i                                         
	\end{tikzcd}\]
	with every morphism admissible, where each \(\mathcal{T}_i\) is a \(k\)-Tate algebra.
	
	Let \(g',f_1',\ldots,f_n',f_{n+1}'\in\B\) be liftings of \(g,f_1,\ldots,f_n,0\in\A\), generating the unit ideal in \(\B\). To see that such a lifting always exists, let \(f_1',\ldots,f_n',g'\in\B\) be liftings of \(g,f_1,\ldots,f_n\in\A\). Since the latter generate the unit ideal in \(\A\), there exist \(a_0,a_1,\ldots,a_n\in\A\) such that \(a_0g+a_1f_1+\ldots+a_nf_n=1\). For liftings \(b_0,b_1,\ldots,b_n\in\B\) of \(a_0,a_1\ldots,a_n\), we then get that \(b_0g'+b_1f_1'+\ldots+b_nf_n'=1+c\) with \(c\in\ker(\varphi')\). Now take \(f_{n+1}'=-c\) as lifting of \(0\in\A\).
	
	Next, since the Gauss norm on \(k\)-Tate algebras is multiplicative, one easily verifies that \(\prod_i\mathcal{T}_i\) is a uniform \(k\)-Banach algebra, and thus so is \(\B\). Hence, by Proposition \ref{prop:closedrationalidealsgeneral}, the ideal \((g'T_1-f_1',\ldots,g'T_{n+1}-f_{n+1}')\) is closed in \(\B\langle r_1^{-1}T_1,\ldots,r_n^{-1}T_n,T_{n+1}\rangle\). As the image of this ideal under the admissible epimorphism
	\[\B\langle r_1^{-1}T_1,\ldots,r_n^{-1}T_n,T_{n+1}\rangle\to\A\langle r_1^{-1}T_1,\ldots,r_n^{-1}T_n\rangle:T_{n+1}\mapsto0\]
	is the ideal \((gT_1-f_1,\ldots,gT_n-f_n)\), it follows that the latter is also closed in \(\A\langle r_1^{-1}T_1,\ldots,r_n^{-1}T_n\rangle\).
\end{proof}
\begin{corollary}\label{cor:kerneltate}
	Let \(\A\) be a \(k\)-Liu algebra. For any \(f\in\A\), the kernel of the canonical morphism \(\A\langle T,T^{-1}\rangle\to\A\langle f,f^{-1}\rangle\) coincides with \((T-f)\subset\A\langle T,T^{-1}\rangle\).
\end{corollary}
\begin{proof}
	By Proposition \ref{prop:closedrationalidealsliu}, we can write
	\[\A\langle T,T^{-1}\rangle=\frac{\A\langle T,S\rangle}{(TS-1)}\qquad\text{and}\qquad\A\langle f,f^{-1}\rangle=\frac{\A\langle T,S\rangle}{(T-f,fS-1)}.\]
	Hence, it suffices to verify that \((T-f,TS-1)=(T-f,fS-1)\) in \(\A\langle T,S\rangle\). The latter follows directly from the equalities
	\begin{align*}
		fS-1=TS-1-S(T-f)\qquad\text{and}\qquad TS-1=fS-1+S(T-f).
	\end{align*}
\end{proof}
The following base change lemma for Liu algebras will allow us later on to prove properties for Liu algebras by reducing to the strict case.
\begin{lemma}\label{lem:basechangeliu}
	Let \(\ell/k\) be an analytic extension and let \(\A\) be a (strictly) \(k\)-Liu algebra. Then \(\A_\ell=\A\widehat{\otimes}_k\ell\) is a (strictly) \(\ell\)-Liu algebra.
\end{lemma}	
\begin{proof}
	Let \(\A_1,\ldots,\A_n\) be a (strictly) rational atlas of \(\A\). We claim that \(\A_1\widehat{\otimes}_k\ell,\ldots,\A_n\widehat{\otimes}_k\ell\) constitutes a (strictly) rational atlas of \(\A\widehat{\otimes}_k\ell\). By Lemma \ref{lem:ratcombas}, each \(\A_i\widehat{\otimes}_k\ell\) is a (strictly) rational localisation of \(\A\widehat{\otimes}_k\ell\). Clearly, each \(\A_i\widehat{\otimes}_k\ell\) is also (strictly) \(\ell\)-affinoid, and (strictly) rational injectivity of \(\iota\otimes\ell:\A\widehat{\otimes}_k\ell\to\prod_i\A_i\widehat{\otimes}_k\ell\) follows from (strictly) rational injectivity  of \(\iota\) and exactness of \(-\widehat{\otimes}_k\ell\). Next, let \(x\in\spb(\A\widehat{\otimes}_k\ell)\) and denote by \(\varphi:\A\to\A\widehat{\otimes}_k\ell\) the canonical base change morphism. Then \(\spb(\varphi)(x)\in\spb(\A_i)\) for some \(\A_i=\A\langle s_\bullet^{-1}f_\bullet/g\rangle\). Hence, by Proposition \ref{prop:ratdesc}, we get for each \(f_j\) in \(f_\bullet\) that
	\[|\varphi(f_j)(x)|=|f_j(\spb(\varphi)(x))|\leq s_j|g(\spb(\varphi)(x))|=s_j|\varphi(g)(x)|,\]
	and thus \(x\in\spb(\A_i\widehat{\otimes}_k\ell)\). It follows that the sets \(\spb(\A_i\widehat{\otimes}_k\ell)\) cover \(\spb(\A\widehat{\otimes}_k\ell)\).
\end{proof}
\begin{remark}
	When \(\ell/k\) is a topgebraic extension of \(k_r\) for \(r\in\bR_{>0}^n\), the converse of Lemma \ref{lem:basechangeliu} also holds (cf. Proposition \ref{prop:basechangeliufull}).
\end{remark}

\section{Liu domains}\label{secLiuDom}
In this section, we define Liu domains of spectra of \(k\)-Liu algebras as a direct generalisation of affinoid domains from the classical setting and prove some of their expected properties. In particular, we show that rational domains are examples of Liu domains.
\begin{definition}\label{def:liudomain}
	Let \(\A\) be a (strictly) \(k\)-Liu algebra. A \emph{(strictly) Liu domain} of \(\spb(\A)\) is a subset \(U\subset\spb(\A)\) equipped with a morphism \(\iota_U:\A\to\A_U\) of (strictly) \(k\)-Liu algebras such that \(U=\spb(\iota_U)(\spb(\A_U))\), and that \emph{represents all (strictly) \(k\)-Liu morphisms into \(U\)}, i.e. for any morphism of (strictly) \(k\)-Liu algebras \(\varphi:\A\to\B\) we have a commutative diagram
	\[\begin{tikzcd}
		\A \arrow[rd, "\varphi"'] \arrow[r, "\iota_U"] & \A_U \arrow[d, dashed] \\
		& \B                   
	\end{tikzcd}\]
	if and only if \(\spb(\varphi)\) factorises through \(U\), in which case the diagram is unique. In case \(\A_U\) is a (strictly) \(k\)-affinoid algebra, \(U\) is called a \emph{(strictly)} \emph{affinoid domain} of \(\spb(\A)\). \(\A_U\) is called the \emph{coordinate ring} of \(U\).
\end{definition}
As might have been expected, rational domains are examples of Liu domains. In order to show this, we need the following two auxiliary results.
\begin{lemma}\label{lem:spectralradius}
	Let \(\A\) be a \(k\)-Liu algebra. For any \(f\in\A\), there exist a \(C\in\bR_{>0}\) and an \(N\in\bZ_{>0}\) such that \(\norm{f^n}_\A\leq C\rho_\A(f)^n\) for all \(n\geq N\). In particular, if \(f\) is quasinilpotent (i.e. \(\rho_\A(f)=0\)), then \(f\) is nilpotent.
\end{lemma}
\begin{proof}
	By \cite[Proposition 2.1.4(i)]{Berkovich} the claim is true when \(\A\) is \(k\)-affinoid. Hence, let \(\A_1,\ldots,\A_n\) be a rational atlas of \(\A\) and take \(f\in\A\). For each \(i\in\{1,\ldots,n\}\) there now exist \(C_i\in\bR_{>0}\) and \(N_i\in\bZ_{>0}\) such that \(\norm{f^n}_{\A_i}\leq C_i\rho_{\A_i}(f)^n\) for all \(n\geq N_i\). As a result, we get for all \(n\geq\max_{1\leq i\leq n} N_i\) that
	\begin{align*}
		\norm{f^n}_\A&=\max_{1\leq i\leq n}\norm{f^n}_{\A_i} \\
		&\leq\max_{1\leq i\leq n}C_i\rho_{\A_i}(f)^n \\
		&=\max_{1\leq i\leq n}\max_{x\in\spb(\A_i)}C_i|f(x)|^n \\
		&\leq\max_{1\leq i\leq n}C_i\cdot\max_{x\in\spb(\A)}|f(x)|^n \\
		&=\max_{1\leq i\leq n}C_i\cdot\rho_\A(f)^n.
	\end{align*}
	Thus, we can take \(C=\max_{1\leq i\leq n}C_i\) and \(N=\max_{1\leq i\leq n}N_i\).
\end{proof}
\begin{corollary}\label{cor:reltateuniprop}
	Let \(\varphi:\A\to\B\) be a morphism of \(k\)-Liu algebras. Let \(b_1,\ldots,b_n\in\B\) and \(r_1,\ldots,r_n\in\bR_{>0}\) with \(\rho_\B(b_i)\leq r_i\) for all \(i\in\{1,\ldots,n\}\). Then there exists a unique morphism \(\Phi:\A\langle r_1^{-1}T_1,\ldots,r_n^{-1}T_n\rangle\to\B\) of \(k\)-Liu algebras extending \(\varphi\) with \(\Phi(T_i)=b_i\) for all \(i\in\{1,\ldots,n\}\).
\end{corollary}
\begin{proof}
	Using Lemma \ref{lem:spectralradius}, one can copy the argument from \cite[Corollary 2.1.5]{Berkovich}.
\end{proof}
\begin{proposition}\label{prop:ratunivprop}
	Let \(\A\) be a (strictly) \(k\)-Liu algebra and denote \(X=\spb(\A)\). A (strictly) rational domain \(X(r_\bullet^{-1}f_\bullet/g)\) is a (strictly) Liu domain of \(X\) with as coordinate ring the (strictly) rational localisation \(\A\langle r_\bullet^{-1}f_\bullet/g\rangle\).
\end{proposition}
\begin{proof}
	Denote by \(\iota:\A\to\A\langle r_1^{-1}f_1/g,\ldots,r_n^{-1}f_n/g\rangle\) the canonical morphism. It follows directly from Proposition \ref{prop:ratdesc} that \(\spb(\iota)(\spb(\A\langle r_\bullet^{-1}f_\bullet/g\rangle))=X(r_\bullet^{-1}f_\bullet/g)\). To see  that \(\iota:\A\to\A\langle r_\bullet^{-1}f_\bullet/g\rangle\) represents all (strictly) \(k\)-Liu morphisms into \(U\), let \(\varphi:\A\to\B\) be a morphism of (strictly) \(k\)-Liu algebras with \(\spb(\varphi)\) factorising through \(X(r_\bullet^{-1}f_\bullet/g)\), and take \(x\in\spb(B)\). One easily verifies that \(\varphi(g)\) is a unit in \(\B\) and that \(|\varphi(f_i)/\varphi(g)(x)|\leq r_i\) for each \(i\in\{1,\ldots,n\}\). Hence, \(\rho_\B(\varphi(f_i)/\varphi(g))\leq r_i\) for each \(i\in\{1,\ldots,n\}\), so that Corollary \ref{cor:reltateuniprop} yields a unique morphism \(\Phi:\A\langle r_\bullet^{-1}T_\bullet\rangle\to\B\) extending \(\varphi\) with \(\Phi(T_i)=\varphi(f_i)/\varphi(g)\) for all \(i\in\{1,\ldots,n\}\). The ideal \((gT_1-f_1,\ldots,gT_n-f_n)\) belongs to the kernel of \(\Phi\), so by Proposition \ref{prop:closedrationalidealsliu}, \(\varphi\) factorises through \(\A\langle r_\bullet^{-1}f_\bullet/g\rangle\).
\end{proof}

Next, we establish some auxiliary results concerning strictly \(k\)-Liu algebras. For a \(k\)-Banach algebra \(\A\), we will denote by \(\mspec(\A)\) the collection of maximal ideals in \(\A\).

\begin{lemma}\label{lem:finext}
	Assume that \(k\) is non-trivially valued. Let \(\A\) be a strictly \(k\)-Liu algebra and \(\varphi:\B\to\A\) a morphism of \(k\)-algebras. Then for any \(\mathfrak{m}\in\mspec(\A)\), \(\A/\mathfrak{m}\) is a finite extension of \(k\) and \(\varphi^{-1}(\mathfrak{m})\) is a maximal ideal in \(\B\).
\end{lemma}
\begin{proof}
	Recall that maximal ideals in a complete normed ring are always closed \cite[Corollary 1.2.4/5]{BGR}. Hence, we can equip \(\A/\mathfrak{m}\) with its residue norm and take some \(x\in\spb(\A/\mathfrak{m})\subset\spb(\A)\). Now let \(\A'\) be a strictly \(k\)-affinoid strictly rational localisation of \(\A\) such that \(x\in\spb(\A')\) and let \(\mathfrak{m}'\) be a maximal ideal in \(\A'\) containing the image of \(\mathfrak{m}\) under the map \(\A\to\A'\). This induces an injection \(\A/\mathfrak{m}\hookrightarrow\A'/\mathfrak{m}'\). Since \(\A'/\mathfrak{m}'\) is a finite extension of \(k\) by \cite[Corollary 6.1.2/3]{BGR}, it follows that \(\A/\mathfrak{m}\) is as well. As any integral domain of finite dimension over a field is itself a field, we conclude that \(\varphi^{-1}(\mathfrak{m})\) is a maximal ideal in \(\B\).
\end{proof}
\begin{proposition}\label{prop:sprinspb}
	Assume that \(k\) is non-trivially valued. For a strictly \(k\)-Liu algebra \(\A\), we can equip \(\mspec(\A)\) with a canonical topology such that we have a functorial canonical inclusion map \(\mspec(\A)\hookrightarrow\spb(\A)\) which is a homeomorphism onto an everywhere dense subset \(\spr(\A)\) of \(\spb(\A)\).
\end{proposition}
\begin{proof}
	Using the fact from Lemma \ref{lem:finext} that \(\A/\mathfrak{m}_x\) is a finite extension of \(k\) for every \(\mathfrak{m}_x\in\mspec(\A)\), one easily verifies that their residue norms are multiplicative. Hence, we obtain a canonical inclusion map \(\mspec(\A)\hookrightarrow\spb(\A)\) via pullback, and \(\mspec(\A)\) can be equipped with the same canonical topology as is defined in \cite[§7.2.1]{BGR} for strictly \(k\)-affinoid algebras. By Lemma \ref{lem:finext}, a morphism \(\A\to\B\) of strictly \(k\)-Liu algebras induces a canonical map \(\mspec(\B)\to\mspec(\A)\), making the canonical inclusion map above functorial. To see that it is also a homeomorphism onto an everywhere dense subset \(\spr(\A)\) of \(\spb(\A)\), one can follow the same argument as in \cite[Proposition 2.1.15]{Berkovich}.
\end{proof}
\begin{definition}
	Assume that \(k\) is non-trivially valued and let \(\A\) be a strictly \(k\)-Liu algebra. The elements of the everywhere dense subset \(\spr(\A)\subset\spb(\A)\) are called \emph{rigid points}. For a morphism \(\varphi:\A\to\B\) of strictly \(k\)-Liu algebras, we denote the restriction of \(\spb(\varphi)\) to the set of rigid points of \(\B\) by \(\spr(\varphi):\spr(\B)\to\spr(\A)\). This is well-defined due to the functoriality of the canonical inclusion map \(\mspec(\A)\hookrightarrow\spb(\A):\mathfrak{m}_x\mapsto x\).
\end{definition}
\begin{proposition}\label{prop:maxidealprop}
	Assume that \(k\) is non-trivially valued. Let \(\A\) be a strictly \(k\)-Liu algebra and let \(U\) be a strictly Liu domain of \(X=\spb(\A)\). Let \(x\in\spr(\A_U)\). Then
	\begin{enumerate}[label=(\roman*)]
		\item \(\spr(\iota_U)\) is injective and satisfies \(\spr(\iota_U)(\spr(\A_U))=U\cap\spr(\A)\);
		\item the map \(\iota_U\) induces an isomorphism \(\A/\mathfrak{m}_{\spr(\iota_U)(x)}\xrightarrow{\sim}\A_U/\mathfrak{m}_x\);
		\item we have \(\mathfrak{m}_x=\iota_U(\mathfrak{m}_{\spr(\iota_U)(x)})\A_U\).
	\end{enumerate}
\end{proposition}
\begin{proof}
	By Lemma \ref{lem:finext}, we have for an arbitrary element \(y\in U\cap\spr(\A)\) that \(\A/\mathfrak{m}_y\) is a finite extension of \(k\), and thus a strictly \(k\)-Liu algebra which we can equip with a residue norm induced by \(\A\). One can now apply the universal property of the strictly Liu domain \(U\) to the morphism \(\A\to\A/\mathfrak{m}_y\) and give the same argument as in \cite[Proposition 7.2.2/1]{BGR} in case \(U\) is a strictly affinoid domain. In order to copy the argument from \cite[Proposition 7.2.2/1]{BGR} for a general Liu domain \(U\), one also needs to show that \(\A_U\to\A_U/\mathfrak{m}_y\A_U\) is a morphism of strictly \(k\)-Liu algebras. To see that this is the case, let \(V\) be a strictly affinoid domain of \(U=\spb(\A_U)\) containing \(y\) and defined by the morphism \(\A_U\to\A_V\). One easily verifies that \(V\) is also a strictly affinoid domain in \(X\) defined by the composition morphism \(\A\to\A_V\). Hence, we get that
	\[\frac{\A_U}{\mathfrak{m}_y\A_U}\cong\A_U\otimes_\A\frac{\A}{\mathfrak{m}_y}\cong\A_U\otimes_\A\left(\A_V\otimes_\A\frac{\A}{\mathfrak{m}_y}\right)\cong\left(\A_U\otimes_\A\A_V\right)\otimes_\A\frac{\A}{\mathfrak{m}_y}\cong\A_V\otimes_\A\frac{\A}{\mathfrak{m}_y}\cong\frac{\A}{\mathfrak{m}_y}\]
	is a finite extension of \(k\). Equipping \(\A_U/\mathfrak{m}_y\A_U\) with its residue norm, it follows that \(\A_U\to\A_U/\mathfrak{m}_y\A_U\) is a morphism of strictly \(k\)-Liu algebras.
\end{proof}
\begin{remark}
	The version of Proposition \ref{prop:maxidealprop} for strictly \(k\)-affinoid algebras, namely \cite[Proposition 7.2.2/1]{BGR}, does not only treat the maximal ideal \(\mathfrak{m}_x\), but also its powers \(\mathfrak{m}_x^n\) for all \(n\in\bZ_{>0}\). After having proven Corollary \ref{cor:noetherian}, one can also obtain this generalisation here. However, at this stage, we do not yet have an argument why the powers \(\mathfrak{m}_x^n\) should be closed, and if so, why their quotients should define strictly \(k\)-Liu algebras. 
	
	Moreover, when proving that \(\A_U\to\A_U/\mathfrak{m}_y\A_U\) is a morphism of strictly \(k\)-Liu algebras, it was essential that \(U=\spb(\iota_U)(\spb(A_U))\) is part of the definition of (strictly) Liu domains in order to see that \(V\) is also a strictly Liu domain in \(X\). A posteriori, though, this directly follows from Corollary \ref{cor:noetherian} and Corollary \ref{cor:liubounded}(i). This allows us to remove the equality \(U=\spb(\iota_U)(\spb(A_U))\) from the definition of (strictly) Liu domains, and derive it as a property of (strictly) Liu domains afterwards. In the strict case, this is done by reproving Proposition \ref{prop:maxidealprop}, whereas the general case follows by the same base change argument as in \cite[Proposition 2.2.4(i)]{Berkovich} via Lemma \ref{lem:basechangeliu}.
\end{remark}
\begin{proposition}\label{prop:defstrictliu}
	Assume that \(k\) is non-trivially valued. Let \(\A\) be a \(k\)-Banach algebra and let \(\A_1,\ldots,\A_n\) be strictly rational localisations of \(\A\). Then \(\A_1,\ldots,\A_n\) constitute a strictly rational atlas of \(\A\) if and only if
	\begin{enumerate}[label=(\roman*)]
		\item each \(\A_i\) is strictly \(k\)-affinoid;
		\item the rigid rational domains \(\spr(\A_i)\) cover \(\spr(\A)\);
		\item the map \(\iota:\A\to\prod_i\A_i\) is strictly rationally injective.
	\end{enumerate}
\end{proposition}
\begin{proof}
	For the direct implication, it suffices to show that (ii) makes sense and holds. This is done by applying Proposition \ref{prop:maxidealprop}(i) to the rational domains \(\spb(\A_1),\ldots,\spb(\A_n)\). For the converse implication, note that \(\cup_i\spb(\A_i)\) is closed and contains \(\cup_i\spr(\A_i)=\spr(\A)\), which is dense in \(\spb(\A)\) by Proposition \ref{prop:sprinspb}. Hence, we get that \(\cup_i\spb(\A_i)=\spb(\A)\).
\end{proof}
\begin{proposition}\label{prop:liudomprop}
	Let \(\varphi:\A\to\B\) be a morphism of (strictly) \(k\)-Liu algebras and denote \(X=\spb(\A)\).
	\begin{enumerate}[label=(\roman*)]
		\item If \(U\subset X\) is a (strictly) rational (resp. Weierstrass, Laurent) domain, then \(\spb(\varphi)^{-1}(U)\) is a (strictly) rational (resp. Weierstrass, Laurent) domain in \(\spb(\B)\) represented by the morphism \(\B\to\B\widehat{\otimes}_\A\A_U\).
		\item Let \(U,V\subset X\) be (strictly) rational (resp. Weierstrass, Laurent) domains. Then \(U\cap V\) is a (strictly) rational (resp. Weierstrass, Laurent) domain represented by the morphism \(\A\to\A_U\widehat{\otimes}_\A\A_V\). In particular, (strictly) Laurent domains are also (strictly) rational domains.
		\item Let \(U\subset X\) be a (strictly) rational (resp. Weierstrass) domain and \(V\subset U\). Then \(V\) is a (strictly) rational (resp. Weierstrass) domain of \(X\) if and only if it is a (strictly) rational (resp. Weierstrass) domain of \(U\).
	\end{enumerate}
\end{proposition}
\begin{proof}
	This is verified in the same way as in \cite[Proposition 7.2.3/6, Proposition 7.2.3/7, Theorem 7.2.4/2]{BGR}.
\end{proof}
\begin{remark}
	Once we know that the morphism \(\iota_U:\A\to\A_U\) for a Liu domain \(U\) in \(\spb(\A)\) is flat, we can show that the completed tensor product of two \(k\)-Liu algebras is again a \(k\)-Liu algebra (cf. Corollary \ref{cor:tensorliu}), and derive that the claims in Proposition \ref{prop:liudomprop} also hold for general Liu domains. Nevertheless, at this stage, an obvious version of the proposition can already be shown when \(\B\), \(U\) and \(V\) are assumed to be (strictly) \(k\)-affinoid.
\end{remark}

\section{Tate acyclicity for Liu algebras}\label{secTate}
In this section, we define a \(G\)-topology and a structure presheaf on Berkovich spectra of \(k\)-Liu algebras and prove a generalised version of Tate's Acyclicity Theorem, showing that this presheaf is an acyclic sheaf. To recall the definition of a \(G\)-topology, see \cite[Definition 9.1.1/1]{BGR}.
\begin{lemma}\label{lem:structsheaf}
	Let \(\A\) be a (strictly) \(k\)-Liu algebra and denote \(X=\spb(\A)\).
	\begin{enumerate}[label=(\roman*)]
		\item We have a \(G\)-topology \(\mathfrak{T}_{X,rat}\) on \(X\) with (strictly) rational domains as admissible open subsets and finite coverings by (strictly) rational domains as admissible coverings. \(\mathfrak{T}_{X,rat}\) is called the \emph{weak (strictly) rational \(G\)-topology} on \(X\).
		\item The assignment \(\mathcal{O}_X(U)\vcentcolon=\A_U\) for every (strictly) rational domain \(U\) in \(X\) defines a presheaf on \(X\) w.r.t. \(\mathfrak{T}_{X,rat}\).
	\end{enumerate}
\end{lemma}
\begin{proof}
	\begin{enumerate}[label=(\roman*)]
		\item It follows from Proposition \ref{prop:liudomprop} that \(\mathfrak{T}_{X,rat}\) is a well-defined \(G\)-topology on \(X\).
		\item It follows from Proposition \ref{prop:ratunivprop} that the assignment \(\mathcal{O}_X(U)\vcentcolon=\A_U\) is well-defined. Functoriality follows from the defining universal property of Liu domains.
	\end{enumerate}
\end{proof}
\begin{remark}
	Let \(\A\) be a (strictly) \(k\)-Liu algebra with \(X=\spb(\A)\). As is the case for affinoid spaces (cf. \cite[§7.3.2]{BGR} and \cite[§2.3]{Berkovich}), one easily verifies for every \(x\in X\) that the stalk \(\mathcal{O}_{X,x}\) is a local ring with maximal ideal given by \(\{f\in\mathcal{O}_{X,x}\,\,|\,\,|f(x)|=0\}\).	Similarly, a morphism of \(k\)-Liu algebras induces local ring morphism on stalks.
\end{remark}
For convenience, we recall the notions of (strictly) rational and Laurent coverings from \cite[§8.2.2]{BGR} as special cases of admissible coverings in the weak (strictly) rational \(G\)-topology.
\begin{definition}
	Let \(\A\) be a \(k\)-Liu algebra and denote \(X=\spb(\A)\).
	\begin{enumerate}[label=(\roman*)]
		\item A \emph{rational covering} of \(X\) is a covering of the form
		\[\mathfrak{U}=\left\{X\left(\frac{r_if_1}{r_1f_i},\ldots,\frac{r_if_n}{r_nf_i}\right)\,\,\middle|\,\,1\leq i\leq n\right\}\]
		with \(f_1,\ldots,f_n\in\A\) generating the unit ideal in \(\A\) and \(r_1,\ldots,r_n\in\bR_{>0}\).
		\item A \emph{Laurent covering} of \(X\) is a covering of the form
		\[\mathfrak{U}=\left\{X\left(r_1^{-\varepsilon_1}f_1^{\varepsilon_1},\ldots,r_n^{-\varepsilon_n}f_n^{\varepsilon_n}\right)\,\,\middle|\,\,\varepsilon_1,\ldots,\varepsilon_n\in\{1,-1\}\right\}\]
		with \(f_1,\ldots,f_n\in\A\) and \(r_1,\ldots,r_n\in\bR_{>0}\).
	\end{enumerate}
	When \(k\) is non-trivially valued and \(\A\) is a strictly \(k\)-Liu algebra, \(\mathfrak{U}\) is respectively called a \emph{strictly rational covering} and a \emph{strictly Laurent covering} if we can take \(r_1=\ldots=r_n=1\).
\end{definition}
The following lemmata are adaptations of results from \cite[§8.2.2]{BGR} to this setting, of which one can essentially copy the original proofs. We recall them mainly for self-containedness of our argument.
\begin{lemma}\label{lem:bgr8221}
	Let \(\A\) be a (strictly) \(k\)-Liu algebra, \(\mathfrak{U}\) a finite covering of \(X=\spb(\A)\) by (strictly) rational domains, and let \(X'\subset X\) be a (strictly) rational domain. Then \(\mathfrak{U}|_{X'}\) is a finite covering of \(X'\) by (strictly) rational domains. If \(\mathfrak{U}\) is a (strictly) rational or (strictly) Laurent covering, \(\mathfrak{U}|_{X'}\) is a covering of the same type.
\end{lemma}
\begin{proof}
	This follows from Proposition \ref{prop:liudomprop}. It is obvious that (strictly) rational (resp. Laurent) coverings induce coverings of the same type.
\end{proof}
\begin{lemma}\label{lem:ratrefine}
	Let \(\A\) be a (strictly) \(k\)-Liu algebra and let \(\mathfrak{U}\) be a covering of \(X=\spb(\A)\) by (strictly) rational domains. Then there exists a (strictly) rational covering of \(X\) which refines \(\mathfrak{U}\).
\end{lemma}
\begin{proof}
	One can copy the proof of \cite[Lemma 8.2.2/2]{BGR}.
\end{proof}
\begin{lemma}\label{lem:bgr8223}
	Let \(\A\) be a (strictly) \(k\)-Liu algebra and let \(\mathfrak{U}\) be a (strictly) rational covering of \(X=\spb(\A)\). Then there exists a (strictly) Laurent covering \(\mathfrak{B}\) of \(X\) such that, for any \(V\in\mathfrak{B}\), the covering \(\mathfrak{U}|_V\) is a (strictly) rational covering of \(V\), which is generated by units in \(\mathcal{O}_X(V)\).
\end{lemma}
\begin{proof}
	One can copy the proof of \cite[Lemma 8.2.2/3]{BGR}.
\end{proof}
\begin{lemma}\label{lem:bgr8224}
	Let \(\A\) be a (strictly) \(k\)-Liu algebra and let \(\mathfrak{U}\) be a (strictly) rational covering of \(X=\spb(\A)\), which is generated by units \(f_1,\ldots,f_n\in\A\). Then there exists a (strictly) Laurent covering \(\mathfrak{B}\) of \(X\), which is a refinement of \(\mathfrak{U}\).
\end{lemma}
\begin{proof}
	One can copy the proof of \cite[Lemma 8.2.2/4]{BGR}.
\end{proof}
\begin{lemma}\label{lem:bgr8225}
	Let \(\A\) be a (strictly) \(k\)-Liu algebra and let \(\mathscr{F}\) be a presheaf on \(\spb(\A)\) w.r.t. the weak (strictly) rational \(G\)-topology. Assume that (strictly) Laurent coverings are universally \(\mathscr{F}\)-acyclic on \(X\), i.e. for each (strictly) rational domain \(X'\subset X\), all (strictly) Laurent coverings of \(X'\) are \(\mathscr{F}\)-acyclic. Then all coverings of \(X\) by (strictly) rational domains are \(\mathscr{F}\)-acyclic.
\end{lemma}
\begin{proof}
Using the lemmata above, one can copy the proof of \cite[Proposition 8.2.2/5]{BGR}.
\end{proof}

\begin{theorem}\label{thm:acyclic}
	Let \(\A\) be a (strictly) \(k\)-Liu algebra and \(\mathfrak{U}=\{U_i\}_{i\in I}\) a finite covering of \(X=\spb(\A)\) by (strictly) rational domains \(U_i\subset X\). Then \(\mathfrak{U}\) is \(\mathcal{O}_X\)-acyclic w.r.t. the weak (strictly) rational \(G\)-topology.
\end{theorem}
\begin{proof}
	First assume that \(k\) is non-trivially valued and \(\A\) is a strictly \(k\)-Liu algebra. By Lemma \ref{lem:bgr8225}, we can assume \(\mathfrak{U}\) to be a strictly Laurent covering. Moreover, using \cite[Corollary 8.1.4/4]{BGR} and an induction argument, we can reduce further to the case where \(\mathfrak{U}\) is generated by a single function \(f\in\A\), i.e. \(\mathfrak{U}=\{X(f),X(f^{-1})\}\). Thus, we have to show that the augmented Čech complex
	\[0\xrightarrow[]{}\A\xrightarrow[]{}\A\langle f\rangle\times\A\langle f^{-1}\rangle\xrightarrow[]{}\A\langle f,f^{-1}\rangle\xrightarrow[]{}0\]
	is an admissible exact sequence. To see this, one can use Proposition \ref{prop:closedrationalidealsliu} and Corollary \ref{cor:kerneltate} to copy the proof of \cite[Corollary 2.20]{Mihara}. The general case now follows from \cite[Proposition 2.1.2(ii)]{Berkovich} after applying a suitable base change.
\end{proof}

\section{Liu algebras - Liu spaces correspondence}\label{secLiuSpaces}
In this section, we prove the correspondence between \(k\)-Liu algebras and \(k\)-Liu spaces, and deduce some additional expected properties of the former. We start by showing that the weak (strictly) rational \(G\)-topology is slightly finer than the \lq classical\rq\ \(k\)-analytic (strict) \(G\)-topology. For convenience, we recall the definition of slightly finer \(G\)-topologies.
\begin{definition}(\cite[Definition 9.1.2/1]{BGR})
	Let \(X\) be a set admitting \(\mathfrak{T}\) and \(\mathfrak{T}'\) as \(G\)-topologies. \(\mathfrak{T}'\) is called \emph{slightly finer} than \(\mathfrak{T}\), if the following conditions are satisfied.
	\begin{enumerate}[label=(\roman*)]
		\item \(\mathfrak{T}'\) is finer than \(\mathfrak{T}\).
		\item The \(\mathfrak{T}\)-open subsets of \(X\) form a basis for \(\mathfrak{T}'\).
		\item For each \(\mathfrak{T}'\)-covering \(\mathfrak{U}\) of a \(\mathfrak{T}\)-open subset \(U\subset X\), there exists a \(\mathfrak{T}\)-covering which refines \(\mathfrak{U}\).
	\end{enumerate}
\end{definition}
\begin{remark}
	Using the notation from the definition above, \cite[Proposition 9.2.3/1]{BGR} implies that (morphisms of) sheaves on \(\mathfrak{T}\) extend uniquely to (morphisms of) sheaves on \(\mathfrak{T}'\). One easily verifies using (ii) that this extension procedure does not alter sheaf cohomology. 
\end{remark}
\begin{lemma}\label{lem:slightlyfiner}
	Let \(\A\) be a (strictly) \(k\)-Liu algebra and denote \(X=\spb(A)\). Let \(\mathfrak{T}_{X,aff}\) be the \(G\)-topology on \(X\) with finite unions of (strictly) affinoid domains as admissible open subsets and finite coverings by finite unions of (strictly) affinoid domains as admissible coverings. Then \(\mathfrak{T}_{X,aff}\) is slightly finer than \(\mathfrak{T}_{X,rat}\).
\end{lemma}
\begin{proof}
	We check the three conditions from the definition above for a general \(k\)-Liu algebra \(\A\). The proof for the strict case is identical. Since a rational domain \(U\) is also a Liu domain, \(U\) can be written as a finite union of affinoid rational domains in \(U\). By transitivity of rational domains (cf. Proposition \ref{prop:liudomprop}(iii)), the latter are also affinoid rational domains in \(X\). It follows that admissible opens and admissible coverings in \(\mathfrak{T}_{X,rat}\) are also admissible for \(\mathfrak{T}_{X,aff}\), i.e. \(\mathfrak{T}_{X,aff}\) is finer than \(\mathfrak{T}_{X,rat}\). Next, let \(U\) be an affinoid domain in \(X\) and let \(\A_1,\ldots,\A_n\) be a rational atlas of \(\A\). By the Gerritzen-Grauert Theorem \cite[Theorem 1.1]{TemkinGG}, the intersection \(U\cap\spb(\A_i)\) can be written as a finite union of rational domains in \(\spb(\A_i)\) for each \(i\in\{1,\ldots,n\}\). Again by transitivity of rational domains, these are also rational domains in \(X\). Hence, the \(\mathfrak{T}_{X,rat}\)-open subsets of \(X\) form a basis for \(\mathfrak{T}_{X,aff}\). The last condition now also follows easily.
\end{proof}
\begin{theorem}\label{thm:isocatliu}
	The category of (strictly) \(k\)-Liu algebras is anti-equivalent to the category of (strictly) \(k\)-Liu spaces. More precisely, taking Berkovich spectra of (strictly) \(k\)-Liu algebras and taking global sections of structure sheaves on (strictly) \(k\)-Liu spaces induce anti-isomorphisms of categories which are mutually inverse to each other.
\end{theorem}
\begin{proof}
	The proof consists of two steps where we show that the respective functors are well-defined. Using \cite[Corollary 3.17]{MaculanPoineau}, one then easily verifies that both functors are mutually inverse to each other. For convenience, we only consider the case of general \(k\)-Liu spaces and algebras. The proof for the strict case is identical.
	
	\emph{Step 1: Taking Berkovich spectra.} Let \(\A\) be a \(k\)-Liu algebra. We show that \(X=(\spb(\A),\mathcal{O}_X)\) is a \(k\)-Liu space, with \(\mathcal{O}_X\) as defined in Lemma \ref{lem:structsheaf}. First note that \(\spb(\A)\) has a canonical \(k\)-analytic structure. Indeed, \(\spb(\A)\) is a non-empty Hausdorff topological space by \cite[Theorem 1.2.1]{Berkovich}, and its \(k\)-affinoid domains constitute a \(k\)-affinoid atlas on \(X\). By Lemma \ref{lem:slightlyfiner}, the \(k\)-analytic \(G\)-topology on \(X\) is slightly finer than the weak rational \(G\)-topology. As the \(k\)-analytic structure sheaf and \(\mathcal{O}_X\) agree on the weak rational \(G\)-topology, we can assume that \(\mathcal{O}_X\) is the \(k\)-analytic structure sheaf. By Theorem \ref{thm:acyclic}, \(\mathcal{O}_X\) is acyclic and thus universally acyclic when \(k\) is non-trivially valued by \cite[Theorem A.5]{MaculanPoineau}. For an analytic extension \(\ell/k\) when \(k\) is trivially valued, we can first consider a non-trivially valued analytic extension \(k_r/k\), in which case \(\mathcal{O}_{X_{k_r}}\) is still acyclic by Lemma \ref{lem:basechangeliu}. As before, \(\mathcal{O}_{X_{\ell_r}}\) is then also acyclic and one can again apply \cite[Theorem A.5]{MaculanPoineau} to conclude that \(\mathcal{O}_{X_\ell}\) is acyclic. Moreover, for distinct points \(x,y\in X\), there exists by definition an \(f\in\A=\mathcal{O}_X(X)\) such that \(|f(x)|\neq|f(y)|\), so \(X\) is holomorphically separable. Lastly, \(X\) is separated and compact by \cite[Theorem 1.2.1 and Proposition 3.1.5]{Berkovich}. Hence, \(X\) is a \(k\)-Liu space. Using Proposition \ref{prop:liudomprop}(i), it is straightforward to check that a morphism of \(k\)-Liu algebras defines a unique morphism of \(k\)-Liu spaces.
	
	\emph{Step 2: Taking global sections.} Let \((X,\mathcal{O}_X)\) be a \(k\)-Liu space. We show that \(\A=\mathcal{O}_X(X)\) is a \(k\)-Liu algebra. By \cite[Proposition 3.12]{MaculanPoineau}, \(X\) is a finite union of affinoid rational domains \(X_1,\ldots,X_n\). This induces a finite collection of \(k\)-affinoid rational localisations \(\A_1=\mathcal{O}_X(X_1),\ldots,\A_n=\mathcal{O}_X(X_n)\) of \(\A\) whose Berkovich spectra cover \(X\cong\spb(\A)\). The latter homeomorphism is \cite[Corollary 3.17]{MaculanPoineau}, where \(\A\) is equipped with a norm making \(\A\to\prod_{i}\A_i\) admissible. Rational injectivity of \(\A\to\prod_{i}\A_i\) follows directly from \(\mathcal{O}_X\) being a sheaf. It follows from \cite[Corollary 3.7]{Xia} that a morphism of \(k\)-Liu spaces yields a morphism of \(k\)-Liu algebras.
\end{proof}
\begin{corollary}\label{cor:noetherian}
	Let \(\A\) be a (strictly) \(k\)-Liu algebra and \(I\subset\A\) an ideal. Then \(\A\) is Noetherian, \(I\) is closed and \(\A/I\) equipped with its residue norm is again a (strictly) \(k\)-Liu algebra.
\end{corollary}
\begin{proof}
	Noetherianity follows by combining Theorem \ref{thm:isocatliu} and \cite[Proposition 2.6(3)]{MaculanPoineau} (see also \cite[Remark 2.7]{MaculanPoineau}). As a result, \(I\) is closed by \cite[Proposition 3.7.2/2]{BGR}. The quotient \(\A/I\) is a \(k\)-Liu algebra by Theorem \ref{thm:isocatliu} and \cite[Corollary 1.16(1) and Lemma 2.4]{MaculanPoineau}.
\end{proof}
Recall that, for a given \(k\)-Banach algebra \(\A\), a Banach \(\A\)-module \(M\) is called a \emph{finite Banach \(\A\)-module} if it admits an admissible epimorphism \(\A^n\to M\) for some \(n\in\bZ_{>0}\).
\begin{corollary}\label{cor:liubounded}
	Let \(\A\) be a \(k\)-Liu algebra. Then the following hold.
	\begin{enumerate}[label=(\roman*)]
		\item All \(k\)-algebra morphisms from a Noetherian \(k\)-Banach algebra to \(\A\) are bounded. In particular, all \(k\)-algebra morphisms between \(k\)-Liu algebras are bounded.
		\item Every finite \(\A\)-module has a canonical Banach \(\A\)-module structure, and all \(\A\)-linear maps between finite Banach \(\A\)-modules are bounded and admissible.
	\end{enumerate}
\end{corollary}
\begin{proof}
	Claim (i) follows directly by combining Theorem \ref{thm:isocatliu} and \cite[Corollary 3.7]{Xia}.\footnote{Note that in \cite[Corollary 3.7]{Xia}, it was forgotten to add the Noetherianity assumption to the domain of a \(k\)-algebra morphism to \(\A\). However, this is essential in order to be able to use \cite[Proposition 3.7.5/2]{BGR} in their proof.} For claim (ii), one can give essentially the same argument as in \cite[Proposition 2.1.9]{Berkovich} (except for admissibility). To see admissibility, one can use \cite[Proposition 2.1.2(ii)]{Berkovich} to reduce to the non-trivially valued case, and conclude by applying Banach's Open Mapping Theorem.
\end{proof}
The following Corollary is a version of \cite[Lemma 2.4.1]{TemkinPinch} for Liu algebras.
\begin{proposition}\label{prop:basechangeliufull}
	Let \(\ell\) be a topgebraic extension of \(k_r\) for some tuple \(r\in\bR_{>0}^n\). Then a \(k\)-Banach algebra \(\A\) is \(k\)-Liu if and only if \(\A_\ell=\A\widehat{\otimes}_k\ell\) is \(\ell\)-Liu.
\end{proposition}
\begin{proof}
	The direct implication is a special case of Lemma \ref{lem:basechangeliu}. For the converse implication, we can assume that \(\ell\) is non-trivially valued, as otherwise we can replace it by any \(\ell_{r'}\) with \(r'\neq1\). Denote \(X=\spb(\A)\) and \(X_\ell=\spb(\A_\ell)\), and let \(R=X_\ell(s_\bullet^{-1}f_\bullet/g)\) be a rational domain of \(X_\ell\). As \(R\) is compact, one can copy the argument from \cite[Proposition 7.2.4/1]{BGR} to see that there is an \(\varepsilon\in\bR_{>0}\) such that \(R\subset X_\ell(\varepsilon g^{-1})\subset X_\ell\) and \(R\cap X_\ell(\varepsilon^{-1}g)=\varnothing\). Also, one easily verifies rigidity of Weierstrass and Laurent domains in compact \(k\)-analytic spaces by the same argument as in \cite[Proposition 7.2.3/3]{BGR}. Hence, after some small perturbations if necessary, we can assume that \(f_\bullet\) and \(g\) are in \(\A_F\) for a finite Galois extension \(F/k_r\). Combining this with \cite[Lemma 2.4.1]{TemkinPinch}, it follows that every rational atlas of \(\A_\ell\) descends to a rational atlas of \(\A_F\) for a finite Galois analytic extension \(F/k_r\). We can thus assume that \(\ell/k_r\) is a finite Galois analytic extension, for which we denote \(G=\Gal(\ell/k)\).
	
	Now fix a rational atlas \(\A_{\ell,1},\ldots,\A_{\ell,n}\) for \(\A_\ell\) and let \(\A_i=\A_{\ell,i}^G\) for every \(i\in\{1,\ldots,n\}\). By \cite[Proposition 2.1.14(ii)]{Berkovich}, each \(\A_i\) is \(k\)-affinoid. One easily verifies that the \(\spb(\A_i)\)'s define a \(k\)-affinoid atlas on \(X\), endowing \(X\) with a \(k\)-analytic structure compatible with that on \(X_\ell\). It now follows from \cite[Theorem A.5]{MaculanPoineau} and Theorem \ref{thm:isocatliu} that \(\A\) is a \(k\)-Liu algebra.
\end{proof}
\begin{remark}
	We can give the same remarks about Proposition \ref{prop:basechangeliufull} as is done for its affinoid version in \cite[Remark 2.4.2]{TemkinPinch}.
	\begin{enumerate}[label=(\roman*)]
		\item One can remove the assumption that the tuple \(r\) is finite. However, it is important for the proof that \(k_r/k\) has an orthogonal Schauder basis.
		\item The corollary does not hold for an arbitrary analytic extension \(\ell/k\). Indeed, let \(x\in\bA_k^1\) be a type 4 point of radius \(r\in\bR_{>0}\). Then its completed residue field \(\ell=\mathcal{H}(x)\) is not \(k\)-affinoid, whereas \(\ell\widehat{\otimes}_k\ell\cong\ell\langle r^{-1}T\rangle\) is \(\ell\)-affinoid. Since \(\spb(\ell)\) consists of only one point, \(\ell\) also cannot be \(k\)-Liu.
	\end{enumerate}
\end{remark}
\begin{proposition}\label{prop:flat}
	Let \(\A\) be a \(k\)-Liu algebra and \(U\) a Liu domain in \(X=\spb(\A)\). Then the restriction map \(\A\to\A_U\) is flat.
\end{proposition}
\begin{proof}
	When \(X\) and \(U\) are strictly \(k\)-Liu, we can use Corollary \ref{cor:noetherian} and \cite[Proposition 2.4]{LiuTohoku} to give the same argument as in \cite[Corollary 7.3.2/6]{BGR}. Using Corollary \ref{cor:liubounded}(ii), one now proves the general case via the same base change argument as in \cite[Proposition 2.2.4]{Berkovich}.
\end{proof}
\begin{corollary}\label{cor:faithful}
	Let \(\A\) be a \(k\)-Banach algebra. Then \(\A\) is a (strictly) \(k\)-Liu algebra if and only if there exist finitely many (strictly) rational localisations \(\A_1,\ldots,\A_n\) of \(\A\) such that
	\begin{enumerate}[label=(\roman*)]
		\item each \(\A_i\) is (strictly) \(k\)-affinoid;
		\item the canonical map \(\iota:\A\to\prod_i\A_i\) is faithfully flat.
	\end{enumerate}
\end{corollary}
\begin{proof}
	By combining Lemma \ref{lem:ratcombas}, Proposition \ref{prop:basechangeliufull}, \cite[Lemma 2.4.1]{TemkinPinch}, and the fact that field extensions are faithfully flat, it suffices to prove the strict case. Hence, assume that \(k\) is non-trivially valued and let \(\A_1,\ldots,\A_n\) be strictly \(k\)-affinoid strictly rational localisations of \(\A\). If \(\A_1,\ldots,\A_n\) constitute a strictly rational atlas of \(\A\), then faithful flatness of \(\iota:\A\to\prod_i\A_i\) follows from Proposition \ref{prop:defstrictliu}(ii) and Proposition \ref{prop:flat} via \cite[Proposition I.3.5/9(e)]{BourbakiCom}. Conversely, if \(\iota:\A\to\prod_i\A_i\) is faithfully flat, then the rigid rational domains \(\spr(\A_i)\) cover \(\spr(\A)\) by \cite[Proposition I.3.5/9(e)]{BourbakiCom} and \(\iota:\A\to\prod_i\A_i\) is rationally injective by \cite[Proposition I.3.5/9(c)]{BourbakiCom} and \cite[Proposition 1.1.9/5]{BGR}. Hence, \(\A_1,\ldots,\A_n\) constitutes a strictly rational atlas of \(\A\) by Proposition \ref{prop:defstrictliu}.	
\end{proof}
\begin{corollary}\label{cor:tensorliu}
	Let \(\A\to\B\) and \(\A\to\C\) be bounded morphisms of \(k\)-Banach algebras with \(\B\) and \(\C\) (strictly) \(k\)-Liu. Then the completed tensor product \(\B\widehat{\otimes}_\A\C\) is again a (strictly) \(k\)-Liu algebra. As a result, the intersection of two Liu domains in a \(k\)-Liu space is again a Liu domain.
\end{corollary}
\begin{proof}
	Let \(\B\langle r_\bullet^{-1}f_\bullet/g\rangle\) and \(\C\langle r_\bullet'^{-1}f_\bullet'/g'\rangle\) be (strictly) \(k\)-affinoid (strictly) rational localisations of \(\B\) and \(\C\). Then
	\[\left(\B\widehat{\otimes}_\A\C\right)\left\langle(r_\bullet r_\bullet')^{-1}\frac{(f_\bullet,g)\otimes (f_\bullet',g')}{g\otimes g'}\right\rangle\cong \B\left\langle r_\bullet^{-1}f_\bullet/g\right\rangle\widehat{\otimes}_\A\C\left\langle r_\bullet'^{-1}f_\bullet'/g'\right\rangle\]
	is also (strictly) \(k\)-affinoid. Now let \(\B_1,\ldots,\B_n\) and \(\C_1,\ldots,\C_m\) be (strictly) rational atlases for \(\B\) and \(\C\) respectively. We claim that
	\[\left\{\B_i\widehat{\otimes}_\A\C_j\,\,\middle|\,\,1\leq i\leq n\text{ and }1\leq j\leq m\right\}\]
	constitutes a (strictly) rational atlas of \(\B\widehat{\otimes}_\A\C\). Indeed, by the observation above, each \(\B_i\widehat{\otimes}_\A\C_j\) is a (strictly) \(k\)-affinoid (strictly) rational localisation of \(\B\widehat{\otimes}_\A\C\). By Corollary \ref{cor:faithful}, it now suffices to show that the canonical morphism \(\B\widehat{\otimes}_\A\C\to\prod_{i,j}\B_i\widehat{\otimes}_\A\C_j\) is faithfully flat. By \cite[Remark I.3.4/(2)]{BourbakiCom} and \cite[Proposition 1.1.9/5]{BGR}, this follows if we show that both morphisms in the composition
	\[\B\otimes_\A\C\to\left(\prod_i\B_i\right)\otimes_\A\C\to\left(\prod_i\B_i\right)\otimes_\A\left(\prod_j\C_j\right)\cong\prod_{i,j}\B_i\otimes_\A\C_j,\]
	are faithfully flat. To see this, one applies \cite[Proposition I.3.2/4]{BourbakiCom} with \((R,S,E,F)=\left(\B,\B\otimes_\A\C,\prod_i\B_i,\B\otimes_\A\C\right)\) and \((R,S,E,F)=\left(\C,\left(\prod_i\B_i\right)\otimes_\A\C,\prod_j\C_j,\left(\prod_i\B_i\right)\otimes_\A\C\right)\).
\end{proof}

\section{Criterion for affinoidness}\label{secAffcrit}
In this section, we prove a criterion that distinguishes (strictly) \(k\)-affinoid algebras within the category of (strictly) \(k\)-Liu algebras, thus answering a conjecture by Temkin (cf. Remark \ref{rem:conjtemkin}). Recall the definition of (graded) reduction from Definition \ref{def:reduction}. For the definition of stable fields, we refer to \cite[§3.6]{BGR}. The nilradical of a ring \(\A\) will be denoted by \(\rad(\A)\).
\begin{lemma}\label{lem:reduced}
	Let \(\A\) be a (strictly) \(k\)-Liu algebra. If \(\A/\rad(\A)\) is (strictly) \(k\)-affinoid, then \(\A\) is also (strictly) \(k\)-affinoid.
\end{lemma}
\begin{proof}
	By \cite[Lemma 2.4.1]{TemkinPinch}, it suffices to consider the strict case, so assume that \(k\) is non-trivially valued and \(\A/\rad(\A)\) is strictly \(k\)-affinoid. Since \(\A\) is Noetherian (cf. Corollary \ref{cor:noetherian}), we can write the nilradical of \(\A\) as \(\rad(\A)=(a_1,\ldots,a_n)\). By assumption, we have an admissible epimorphism
	\[\varphi:k\langle X_1,\ldots,X_m\rangle\to\A/\rad(\A).\]
	We claim that \(\varphi\) can be lifted to an admissible epimorphism
	\[\varphi':k\langle X_1,\ldots,X_m,Y_1,\ldots,Y_n\rangle\to\A\]
	sending \(Y_j\) to \(a_j\) for every \(j\in\{1,\ldots,n\}\). Indeed, for every \(i\in\{1,\ldots,m\}\), we can pick \(b_i\in\A\) such that \(\varphi(X_i)=b_i+\rad(\A)\), and fix \(\varphi'(X_i)=b_i\). Note that \(\varphi'\) is well-defined by Corollary \ref{cor:reltateuniprop}, since \(\rho_\A(a)=\rho_{\A/\rad(\A)}(a+\rad(\A))\) for every \(a\in\A\). To see that \(\varphi'\) is surjective, take \(a\in \A\) and \(f\in k\langle X_1,\ldots,X_m\rangle\) such that \(\varphi(f)=a+\rad(\A)\). Then there are elements \(c_1,\ldots,c_n\in\A\) such that
	\begin{align*}
		a&=\varphi'(f)+c_1a_1+\ldots+c_na_n \\
		 &=\varphi'(f)+c_1\varphi'(Y_1)+\ldots+c_n\varphi'(Y_n).
	\end{align*}
	One now repeats this procedure inductively to rewrite the \(c_i\)'s until the powers of each \(\varphi(Y_i)\) start to vanish. Admissibility of \(\varphi'\) follows from Corollary \ref{cor:liubounded} and Banach's Open Mapping Theorem.
\end{proof}
\begin{theorem}\label{thm:affcrit}
	Let \(k\) be stable and \(\A\) a \(k\)-Liu algebra.
	\begin{enumerate}[label=(\roman*)]
		\item \(\A\) is \(k\)-affinoid if and only if \(\widetilde{\A}_\gr\) is a finitely generated \(\widetilde{k}_\gr\)-algebra and there exist elements \(f_1,\ldots,f_n\in\A\), with spectra radii \(r_1,\ldots,r_n\), such that \(\A\langle r_1^{-1}f_1^{-1}\rangle,\ldots,\A\langle r_n^{-1}f_n^{-1}\rangle\) forms a rational atlas of \(\A\).
		\item If \(\A\) is strictly \(k\)-Liu, then \(\A\) is strictly \(k\)-affinoid if and only if \(\widetilde{\A}\) is a finitely generated \(\widetilde{k}\)-algebra and there exist elements \(f_1,\ldots,f_n\in\A\), each with spectral radius \(1\), such that \(\A\langle f_1^{-1}\rangle,\ldots,\A\langle f_n^{-1}\rangle\) forms a strictly rational atlas of \(\A\).
	\end{enumerate}
\end{theorem}
\begin{proof}
	By (the proof of) Lemma \ref{lem:basechangeliu}, \cite[Lemma 2.4.1]{TemkinPinch} and \cite[Corollary 1.3]{TemkinGraded}, we can assume that \(k\) is non-trivially valued, \(|k^\times|\) is divisible and \(\A\) is a strictly \(k\)-Liu algebra. In particular, it now follows from \cite[Corollary 2.1.6]{Berkovich} that \(\rho_\A(\A)=\rho_\A(k)\), so \(\widetilde{\A}_\gr\cong\widetilde{\A}\otimes_{\widetilde{k}}\widetilde{k}_\gr\) by \cite[Lemma 1.1]{TemkinGraded}. By \cite[Corollary 1.3]{TemkinGraded}, it thus suffices to prove (ii), in which case the direct implication follows by \cite[Corollary 6.3.4/3]{BGR} and taking \(f_1=1\).
	
	For the converse implication of (ii), we can assume by Lemma \ref{lem:reduced} that \(\A\) is reduced. Let \(\widetilde{\varphi}:\widetilde{k}[x_1,\ldots,x_m]\to\widetilde{\A}\) be a \(\widetilde{k}\)-algebra epimorphism and denote by \(a_1,\ldots,a_m\in\A\) liftings of \(\widetilde{\varphi}(x_1),\ldots,\widetilde{\varphi}(x_n)\) respectively. Next, by assumption, there are elements \(f_1,\ldots,f_n\in\A\) such that \(\A\langle f_1^{-1}\rangle,\ldots,\A\langle f_n^{-1}\rangle\) forms a strictly rational atlas of \(\A\) with \(\rho_A(f_i)=1\) for every \(i\in\{1,\ldots,n\}\). Consider the \(k\)-Liu algebra morphism
	\[\varphi:\B=k\left\langle X_1,\ldots,X_m,Y_1,\ldots,Y_n\right\rangle\to\A:\left\{ \begin{aligned} 
		X_i\mapsto a_i&\text{ for }i\in\{1,\ldots,m\} \\
		Y_j\mapsto f_j&\text{ for }j\in\{1,\ldots,n\}
	\end{aligned} \right.,\]
	which is well-defined by Corollary \ref{cor:reltateuniprop}. By an obvious generalisation of \cite[Proposition 7.2.6/3]{BGR} to strictly \(k\)-Liu algebras, \(\varphi\) induces \(\widetilde{k}\)-algebra epimorphisms
	\[\widetilde{\varphi}_j:\left(\B\left\langle Y_j^{-1}\right\rangle\right)^\sim=\widetilde{k}\left[x_1,\ldots,x_m,y_1,\ldots,y_n,y_j^{-1}\right]\to\left(\A\left\langle f_j^{-1}\right\rangle\right)^\sim=\widetilde{\A}\left[\widetilde{f}_j^{-1}\right]\]
	for every \(j\in\{1,\ldots,n\}\). It now follows directly from \cite[Proposition 6.4.2/1]{BGR} that the corresponding \(k\)-Liu algebra morphisms
	\[\varphi_j:\B\left\langle Y_j^{-1}\right\rangle\to\A\left\langle f_j^{-1}\right\rangle\]
	are surjective for every \(j\in\{1,\ldots,n\}\). As \(\spb(\varphi)(\spb(\A))\) is contained in \(\cup_j\spb(\B\langle Y_j^{-1}\rangle)\), it now follows from the fact that \(\spb(\B)\) is cohomologically Stein that \(\varphi:\B\to\A\) is an admissible epimorphism, i.e. \(\A\) is strictly \(k\)-affinoid.
\end{proof}

\section{Berkovich spectra of Fr\'echet algebras}\label{secFrech}
To extend our algebraic description of (strictly) \(k\)-Liu spaces to general (strictly) \(k\)-Stein spaces, we must define Berkovich spectra more broadly, namely for \(k\)-Fr\'echet algebras rather than \(k\)-Banach algebras. While this may be well known to experts, no suitable source for this could be found in the literature, so we recall it here for convenience.
\begin{definition}
	A \emph{\(k\)-Fr\'echet algebra} is a \(k\)-algebra \(\A\) which is also a \(k\)-Fr\'echet space, i.e. it is a complete and Hausdorff topological \(k\)-vector space whose topology is defined by a fixed countable family of \(k\)-seminorms \(\{\norm{\cdot}_i\}_{i\in\bZ_{\geq0}}\), and for which multiplication is jointly continuous. The \emph{category of \(k\)-Fr\'echet algebras} consists of \(k\)-Fr\'echet algebras and bounded \(k\)-algebra morphisms between them (cf. Definition \ref{def:frechetbounded}(i)).
\end{definition}
\begin{remark}
	More concretely, a subset \(U\) of a \(k\)-Fr\'echet space \(\A\), of which the topology is defined by a countable family of seminorms \(\{\norm{\cdot}_i\}_{i\in\bZ_{\geq0}}\), is open if and only if for every \(f\in U\) there is a finite subset \(N\subseteq\bZ_{\geq0}\) and an \(r\in\bR_{>0}\) such that
	\[\left\{g\in\A \,\,|\,\, \norm{f-g}_n< r\text{ for every }n\in N\right\}\subseteq U.\]
	In particular, a sequence \((f_n)_{n\in\bZ_{\geq0}}\) in \(\A\) converges to zero if and only if \(\lim_{n\to+\infty}\norm{f_n}_i=0\) for every \(i\in\bZ_{\geq0}\).
	
	In case the family of seminorms \(\{\norm{\cdot}_i\}_{i\in\bZ_{\geq0}}\) on a \(k\)-Fr\'echet algebra \(\A\) is increasing, i.e. \(\norm{f}_i\leq\norm{f}_j\) for every \(f\in\A\) whenever \(i\leq j\), joint continuity of multiplication is equivalent to there being, for every \(n\in\bZ_{\geq0}\), a constant \(C_n\in\bR_{>0}\) and an integer \(m\geq n\) such that \(\norm{fg}_{n}\leq C_n\norm{f}_m\norm{g}_m\) for all \(f,g\in\A\). When we can take \(n=m\), we can assume the seminorms to be submultiplicative by \cite[Proposition 1.2.1/2]{BGR}. All of this will be the case for the \(k\)-Fr\'echet algebras that we will be interested in, namely \(k\)-Stein algebras (cf. Definition \ref{def:steinalg}).
\end{remark}
\begin{definition}\label{def:frechetbounded}
	Let \((\A,\{\norm{\cdot}_{\A,i}\}_{i\in\bZ_{\geq0}})\) and \((\B,\{\norm{\cdot}_{\B,i}\}_{i\in\bZ_{\geq0}})\) be \(k\)-Fr\'echet algebras.
	\begin{enumerate}[label=(\roman*)]
		\item A \(k\)-algebra morphism \(\varphi:\A\to\B\) is called \emph{bounded} if there exists for each \(n\in\bZ_{\geq0}\) a constant \(C_n\in\bR_{>0}\) and a non-empty finite subset \(N_n\subseteq\bZ_{\geq0}\) such that we have for all \(f\in\A\) that
		\[\norm{\varphi(f)}_{\B,n}\leq C_n\max_{i\in N_n}\norm{f}_{\A,i}.\]
		\item A seminorm \(\norm{\cdot}\) on \(\A\) is called \emph{bounded} if there exists a non-empty finite subset \(N\subseteq\bZ_{\geq0}\) and a constant \(C\in\bR_{>0}\) such that for every \(f\in\A\) we have that
		\[\norm{f}\leq C\max_{i\in N}\norm{f}_i.\]
	\end{enumerate}
\end{definition}
\begin{remark}
	Bounded morphisms and seminorms are also continuous. Indeed, Fr\'echet spaces have an equivalent description as metric spaces, so continuity is equivalent to sequential continuity, which can be easily verified. Conversely, when \(k\) is non-trivially valued, continuity of morphisms or seminorms also implies boundedness by an argument similar to that of the analogous statement for \(k\)-Banach algebras (cf. \cite[Proposition 2.1.8/2]{BGR}). When \(k\) is trivially valued, this is not even necessarily the case for \(k\)-Banach algebras; take for example the identity map \((k[T],\norm{\cdot}_1)\to(k[T],\norm{\cdot}_2)\) where \(\norm{T}_1=1\) and \(\norm{T}_2>1\). 
\end{remark}
\begin{definition}
	Let \(\A\) be a \(k\)-Fr\'echet algebra. The \emph{Berkovich spectrum} of \(\A\) is the set
	\[\spb(\A)\vcentcolon=\left\{\text{bounded multiplicative seminorms }|\cdot(x)|:\A\to\bR_{\geq0}\right\},\]
	equipped with the weakest topology such that \(\spb(\A)\to\bR_{\geq0}:|\cdot(x)|\mapsto|f(x)|\) is continuous for every \(f\in\A\).
\end{definition}
\begin{remark}\label{rem:characters}
	 As is done in \cite[Definition 4.1]{BBBK}, one can easily rephrase the definition above in terms of characters, i.e. \(\spb(\A)\) consists of all equivalence classes of bounded morphisms from \(\A\) to an analytic extension of \(k\). Moreover, a bounded morphism \(\varphi:\A\to\B\) between \(k\)-Fr\'echet algebras functorially induces a continuous map
	\[\spb(\varphi):\spb(\B)\to\spb(\A):|\cdot(x)|\mapsto|\varphi(\cdot)(x)|.\]
	
	Moreover, just as is the case for \(k\)-Banach algebras, the requirement for the family of seminorms to be part of the definition of a given \(k\)-Fr\'echet algebra, is essential in order to get a well-defined Berkovich spectrum when \(k\) is trivially valued. 
\end{remark}
The following proposition is a variation on \cite[Theorem 1.2.1]{Berkovich} for \(k\)-Fr\'echet algebras.
\begin{proposition}\label{prop:frechetspec}
	Let \((\A,\{\norm{\cdot}_i\}_{i\in\bZ_{\geq0}})\) be a \(k\)-Fr\'echet algebra. Then \(\spb(\A)\) is a non-empty Hausdorff space countable at infinity.
\end{proposition}
\begin{proof}
	The Hausdorff property is shown in the same way as for \(k\)-Banach algebras in \cite[Theorem 1.2.1]{Berkovich}. Furthermore, for every \(n\in\bZ_{\geq0}\), we can define a \(k\)-Banach algebra \(\A_n\) as the completion of \(\A\) w.r.t. the seminorm \(\max_{0\leq i\leq n}\norm{\cdot}_i\). We then get continuous mappings \(\A\to\A_n\) with dense image, inducing mutually compatible continuous injections \(\iota_n:\spb(\A_n)\hookrightarrow\spb(\A)\). Since \(\spb(\A_n)\) is compact and \(\spb(\A)\) is Hausdorff, each \(\iota_n\) is a homeomorphism onto its image, so that we can write \(\spb(\A)=\cup_{i\in\bZ_{\geq0}}\spb(\A_i)\) (topologically). It now follows from \cite[Theorem 1.2.1]{Berkovich} that \(\spb(\A)\) is non-empty and countable at infinity.
\end{proof}
\section{Stein algebras - Stein spaces correspondence}\label{secStein}
In this section, we introduce \(k\)-Stein algebras as a specific collection of \(k\)-Fr\'echet algebras, obtained as suitable inverse limits of \(k\)-Liu algebras, and show that their Berkovich spectra coincide exactly with \(k\)-Stein spaces. We also prove some useful properties of \(k\)-Stein algebras.

\begin{definition}\label{def:steinalg}
	A \emph{(strictly) \(k\)-Stein algebra} is a \(k\)-Fr\'echet algebra \(\A\) that can be obtained as the inverse limit in the category of \(k\)-Fr\'echet algebras of a system of (strictly) \(k\)-Liu algebras \((\A_i)_{i\in\bZ_{\geq0}}\) where for every \(i\in\bZ_{\geq0}\) the morphism \(\A_{i+1}\to\A_i\) has dense image and defines a Liu domain such that \(\spb(\A_i)\) is a neighbourhood of \(\spb(\A_{i-1})\) in \(\spb(\A_{i+1})\). For such a system of (strictly) \(k\)-Liu algebras, we call \(\varprojlim_{i}\A_i\) a \emph{presentation} of \(\A\). We define the \emph{category of (strictly) \(k\)-Stein algebras} as the category consisting of (strictly) \(k\)-Stein algebras with bounded \(k\)-algebra morphisms between them.
\end{definition}
\begin{remark}
	To obtain an explicit algebraic description of (strictly) Liu domain embeddings, let \(\A\) be a (strictly) \(k\)-Liu algebra and recall that (strictly) Liu domains can be written as finite unions of (strictly) affinoid (strictly) rational domains \cite[Proposition 3.12]{MaculanPoineau}. Hence, a (strictly) Liu domain embedding in \(\spb(\A)\) corresponds to a morphism of (strictly) \(k\)-Liu algebras \(\varphi:\A\to\B\) for which there exists a (strictly) rational atlas \(\B_1,\ldots,\B_n\) of \(\B\) where each \(\B_i\) is of the form \(\B\langle r_\bullet^{-1}\varphi(f_\bullet)/\varphi(g)\rangle\cong\A\langle r_\bullet^{-1}f_\bullet/g\rangle\). Moreover, it is essential in Definition \ref{def:steinalg} to ask for each \(\spb(\A_i)\) to be a neighbourhood of \(\spb(\A_{i-1})\) in \(\spb(\A_{i+1})\) in order to ensure that \(\spb(\A)\) will be a locally compact topological space. Lastly, when \(k\) is non-trivially valued, \(k\)-Stein algebras are examples of \(k\)-Fr\'echet-Stein algebras in the sense of \cite{SchneiderTeitelbaum} by Corollary \ref{cor:noetherian} and Proposition \ref{prop:flat}.
\end{remark}

To prove the following correspondence between (strictly) \(k\)-Stein algebras and (strictly) \(k\)-Stein spaces, we essentially follow the same arguments as in \cite[§4.2]{BBBK}.
\begin{theorem}\label{thm:isocatstein}
	The category of (strictly) \(k\)-Stein algebras is anti-equivalent to the category of (strictly) \(k\)-Stein spaces. More precisely, taking Berkovich spectra of (strictly) \(k\)-Stein algebras and taking global sections of structure sheaves on (strictly) \(k\)-Stein spaces induce anti-isomorphisms of categories which are mutually inverse to each other.
\end{theorem}
\begin{proof}
	The proof consists of two steps where we show that the respective functors are well-defined. By construction and Theorem \ref{thm:isocatliu}, it will be clear that both functors are mutually inverse to each other. For convenience, we only consider the case of general \(k\)-Stein algebras and spaces. The proof for the strict case is identical.

	\emph{Step 1: Taking Berkovich spectra.} Let \(\A=\varprojlim_i\A_i\) be a \(k\)-Stein algebra. By Proposition \ref{prop:frechetspec} and its proof, \(\spb(\A)\) is a non-empty Hausdorff topological space for which we can write \(\spb(\A)=\cup_{i\in\bZ_{\geq0}}\spb(\A_i)\). By the definition of \(k\)-Stein algebras, each \(\spb(\A_{i+1})\) is a neighbourhood of \(\spb(\A_i)\) with compatible restriction maps for their \(k\)-analytic structure. Hence, one obtains a \(k\)-analytic structure on \(\spb(\A)\) by glueing. It follows from the characterisation of \(k\)-Stein spaces as being \(W\)-exhausted by Liu domains that \(\spb(\A)\) is a \(k\)-Stein space.
	
	Now let \(\varphi:\A\to\B=\varprojlim_j\B_j\) be a morphism of \(k\)-Stein algebras, inducing a continuous map \(\spb(\varphi):\spb(\B)\to\spb(\A)\). As \(\varphi\) is bounded, each morphism \(\A\to\B_j\) factorises through some \(\A_i\) with \(i\in\bZ_{\geq0}\). Hence, possibly after re-indexing the systems in a suitable way, we get compatible morphisms of \(k\)-Liu spaces \(\spb(\varphi_i):\spb(\B_i)\to\spb(\A_i)\) for every \(i\in\bZ_{\geq0}\), which commute with \(\spb(\varphi)\). Thus, \(\spb(\varphi)\) defines a morphism of \(k\)-Stein spaces. One easily verifies this to be functorial. In particular, it follows that the \(k\)-analytic structure on \(\spb(\A)\) does not depend on the chosen presentation of \(\A\).
	
	\emph{Step 2: Taking global sections.} Next, let \(X\) be a \(k\)-Stein space with \(\{D_i\}_{i\in\bZ_{\geq0}}\) a \(W\)-exhaustion of \(X\) by Liu domains. Possibly after deleting some \(D_i\)'s, we can assume that, for every \(i\in\bZ_{>0}\), \(D_i\) is a neighbourhood of \(D_{i-1}\) within \(D_{i+1}\). One easily verifies that \(\mathcal{O}_X(X)\cong\varprojlim_i\mathcal{O}_X(D_i)\), endowing it with a clear \(k\)-Fr\'echet structure and presentation as \(k\)-Stein algebra with \(\spb(\mathcal{O}_X(X))\cong X\).
	
	Lastly, let \(Y=\spb(\B)\to X=\spb(\A)\) be a morphism of \(k\)-Stein spaces. It remains to show that the induced \(k\)-algebra morphism \(\varphi:\A\to\B\) is bounded. Note that for every \(j\in\bZ_{\geq0}\), the image of \(\spb(\B_j)\) must be contained in some \(\spb(\A_i)\) with \(i\in\bZ_{\geq0}\). Possibly after re-indexing the systems in a suitable way, we thus get a system of \(k\)-algebra morphisms \(\varphi_i:\A_i\to\B_i\), which are bounded by Corollary \ref{cor:liubounded}(i), with \(\varprojlim_i\varphi_i=\varphi\). Hence, \(\varphi:\A\to\B\) is bounded. Given two \(W\)-exhaustions of \(X\), one can apply a similar argument to show that (up to equivalence) the \(k\)-Fr\'echet structure on \(\mathcal{O}_X(X)\) does not depend on the chosen \(W\)-exhaustion of \(X\).
\end{proof}

\begin{lemma}\label{lem:maxfingen}
	Assume that \(k\) is non-trivially valued. Let \(\varphi:\A=\varprojlim_i\A_i\to\B\) be a morphism of strictly \(k\)-Stein algebras with \(\B\) strictly \(k\)-Liu. Let \(\mathfrak{n}\) be a maximal ideal in \(\B\) and denote \(\mathfrak{m}=\varphi^{-1}(\mathfrak{n})\). Then we have that
	\begin{enumerate}[label=(\roman*)]
		\item \(\mathfrak{m}\) is a maximal ideal in \(\A\) with \(\A/\mathfrak{m}\) a finite extension of \(k\).
	\end{enumerate}
	Assume now that \(\varphi:\A\to\B\) defines a strictly Liu domain. Then also the following hold.
	\begin{enumerate}[label=(\roman*)]\setcounter{enumi}{1}
		\item For every \(n\in\bZ_{>0}\), \(\varphi:\A\to\B\) induces an isomorphism \(\A/\mathfrak{m}^n\xrightarrow{\sim}\B/\mathfrak{n}^n\).
		\item \(\mathfrak{n}=\mathfrak{m}\B\).
	\end{enumerate}
	In particular, \(\mathfrak{m}\) corresponds to a rigid point \(x\in X=\spb(\A)\) for which one has
	\begin{enumerate}[label=(\roman*)]\setcounter{enumi}{3}
		\item \(\widehat{\A}_{\mathfrak{m}}\cong\widehat{\mathcal{O}}_{X,x}\) for the adic completions.
	\end{enumerate}
\end{lemma}
\begin{proof}
	We can identify \(\A/\mathfrak{m}\) as a subring of \(\B/\mathfrak{n}\), which is a finite extension of \(k\) by Lemma \ref{lem:finext}. Since integral domains which are finite-dimensional \(k\)-algebras, are fields, it follows that \(\mathfrak{m}\) is a maximal ideal in \(\A\) and \(\A/\mathfrak{m}\) is a finite extension of \(k\). In case \(\varphi:\A\to\B\) defines a Liu domain, one can now follow the same argument as in \cite[Proposition 7.2.2/1]{BGR} to show that \(\mathfrak{n}=\mathfrak{m}\B\) and \(\A/\mathfrak{m}^n\cong\B/\mathfrak{m}^n\B\) for any \(n\in\bZ_{>0}\). Lastly, applying (ii) to \(\A\to\A_i\) for \(i\in\bZ_{\geq0}\) sufficiently large, it follows from \cite[Proposition 2.4]{LiuTohoku} that
	\[\widehat{\A}_{\mathfrak{m}}\cong\widehat{(\A_i)}_{\mathfrak{m}\A_i}\cong\widehat{\mathcal{O}}_{X_i,x}=\widehat{\mathcal{O}}_{X,x},\]
	where \(X_i=\spb(\A_i)\).
\end{proof}
The following proposition is a non-Archimedean analytic version of \cite[Theorem 2]{Forster}.
\begin{proposition}\label{prop:maxfingenclosed}
	Assume that \(k\) is non-trivially valued. Let \(\A\) be a strictly \(k\)-Stein algebra. Then any maximal ideal \(\mathfrak{m}\) in \(\A\) is closed if and only if it is finitely generated.
\end{proposition}
\begin{proof}
	Fix a presentation \(\A=\varprojlim_{i}\A_i\) and denote the corresponding projection maps by \(\pi_i:\A\to\A_i\). For both implications, we will consider for every \(i\in\bZ_{\geq0}\) the ideal \(\mathfrak{m}\A_i\) generated by \(\pi_i(\mathfrak{m})\) in \(\A_i\). Then two cases can occur: either there is a \(j\in\bZ_{\geq0}\) such that \(\mathfrak{m}\A_j\) is a proper ideal of \(\A_j\), or \(\mathfrak{m}\A_i=\A_i\) for every \(i\in\bZ_{\geq0}\). We will refer to these cases as \emph{Case (i)} and \emph{Case (ii)} respectively.
	
	Assume that \(\mathfrak{m}\) is closed. In Case (i), we can assume that \(j=0\), and it follows from Lemma \ref{lem:maxfingen}(ii)-(iii) that \(\mathfrak{m}\A_i\) is a maximal ideal in \(\A_i\) with \(\A/\mathfrak{m}^n\cong\A_i/(\mathfrak{m}\A_i)^n\) for all \(i,n\in\bZ_{\geq0}\). Moreover, since each morphism \(\A\to\A_i\) has dense image, the inclusion \(\mathfrak{m}/\mathfrak{m}^2\hookrightarrow\mathfrak{m}\A_i/(\mathfrak{m}\A_i)^2\) of finite-dimensional \(\A/\mathfrak{m}\)-vector spaces is an isomorphism for every \(i\in\bZ_{\geq0}\). Hence, lifting a basis for \(\mathfrak{m}/\mathfrak{m}^2\) to \(\A\) yields elements \(f_1,\ldots,f_n\in\A\), whose images in \(\A_i\) generate \(\mathfrak{m}\A_i\) by Nakayama's lemma for every \(i\in\bZ_{\geq0}\). It now follows directly from the obvious non-Archimedean analogue of \cite[Theorem V.5.4; p.161]{GrauertRemmert} that \(\mathfrak{m}=(f_1,\ldots,f_n)\). In Case (ii), when \(\mathfrak{m}\A_i=\A_i\) for every \(i\in\bZ_{\geq0}\), it follows that \(\mathfrak{m}\) is dense in \(\A\) and thus \(\mathfrak{m}=\A\) since \(\mathfrak{m}\) is closed. This contradicts the assumption of \(\mathfrak{m}\) being a maximal ideal.
	
	Now assume that \(\mathfrak{m}\) is finitely generated. To show that \(\mathfrak{m}\) is closed, we can follow the same strategy as in \cite[Lemma 4.21]{BBBK}. In Case (i), we get that \(\pi_j^{-1}(\mathfrak{m}\A_j)\) equals \(\mathfrak{m}\) since it must be a proper ideal of \(\A\) containing \(\mathfrak{m}\). As \(\pi_j\) is bounded (and thus continuous as \(k\) was assumed to be non-trivially valued), it follows from Corollary \ref{cor:noetherian} that \(\mathfrak{m}\) is closed. In Case (ii), let \(f_1,\ldots,f_n\in\A\) be a set of generators of \(\mathfrak{m}\). As \(\mathfrak{m}\A_i=\A_i\) for every \(i\in\bZ_{\geq0}\), the elements \(\pi_i(f_1),\ldots,\pi_i(f_n)\) have no common zeroes in \(\spb(\A_i)\) for every \(i\in\bZ_{\geq0}\), so \(f_1,\ldots,f_n\) have no common zeroes in \(\spb(\A)\). Using the non-Archimedean analytic version of Cartan's Theorem A \cite[Proposition 2.1]{MaculanPoineau}, one now verifies in exactly the same way as in \cite[Theorems V.5.5 and V.5.4; p.161]{GrauertRemmert} that there exist elements \(g_1,\ldots,g_n\in\A\) such that \(f_1g_1+\ldots+f_ng_n=1\) and thus \(\mathfrak{m}=\A\). This contradicts the assumption of \(\mathfrak{m}\) being a maximal ideal.
\end{proof}
\begin{remark}
	More generally, for a \(k\)-Stein algebras \(\A\), any finitely generated ideal of \(\A\) is closed. When \(k\) is non-trivially valued, this follows from \cite[Corollary 3.4(iv) and Lemma 3.6(iii)]{SchneiderTeitelbaum}. The general case follows by descent via the same argument as in \cite[Proposition 2.1.3]{Berkovich}.
\end{remark}
\begin{corollary}
	Assume that \(k\) is non-trivially valued. Let \(\A=\varprojlim_i\A_i\) be a strictly \(k\)-Stein algebra. Then the collection of rigid points \(\spr(\A)\) in the strictly \(k\)-Stein space \(X=\spb(\A)\) corresponds exactly to the collection of closed maximal ideals in \(\A\).
\end{corollary}
\begin{proof}
	Let \(x\in\spb(\A)\) be a rigid point in \(X\), corresponding to a maximal ideal \(\mathfrak{n}\) in \(\A_i\) for some \(i\in\bZ_{\geq0}\). Since \(\pi_i:\A\to\A_i\) is bounded (and thus continuous as \(k\) is non-trivially valued), it follows from Corollary \ref{cor:noetherian} that \(\pi_i^{-1}(\mathfrak{n})\) is closed. By Lemma \ref{lem:maxfingen}, it is also maximal. Conversely, let \(\mathfrak{m}\) be a closed maximal ideal in \(\A\). Then \(\mathfrak{m}\A_i\) is a maximal ideal in \(\A_i\) for \(i\in\bZ_{\geq0}\) sufficiently large by the proof of Proposition \ref{prop:maxfingenclosed}. Hence, \(\mathfrak{m}\) corresponds to a rigid point \(x\in\spb(\A_i)\).
\end{proof}
The following corollary is a non-Archimedean analytic version of \cite[Theorem 3]{Forster}.
\begin{corollary}\label{cor:algbounded}
	Let \(\A\) and \(\B\) be two strictly \(k\)-Stein algebras. A \(k\)-algebra morphism \(\varphi:\A\to\B\) is continuous if and only if for every finitely generated maximal ideal \(\mathfrak{m}\) in \(\B\) the ideal \(\varphi^{-1}(\mathfrak{m})\) is finitely generated in \(\A\).
\end{corollary}
\begin{proof}
	It suffices to show that \(\varphi_i:\A\to\B_i\) is bounded for every \(i\in\bZ_{\geq0}\). Using the Closed Graph Theorem for Fr\'echet spaces \cite[Corollaire I.3.3/5]{BourbakiVec}, one sees that the proof of the continuity criterion \cite[Proposition 3.7.5/1]{BGR} also holds for \(k\)-Fr\'echet algebras. It now follows from \cite[Proposition 3.6]{Xia} and Proposition \ref{prop:maxfingenclosed} that \(\varphi\) is continuous.
\end{proof}
\begin{remark}
	In the complex analytic setting, one can show that actually any \(\bC\)-algebra morphism between \(\mathbb{C}\)-Stein algebras (associated with finite dimensional \(\bC\)-Stein spaces) satisfies the condition from Corollary \ref{cor:algbounded} \cite[Theorem 5]{Forster}. However, in order to emulate the complex analytic proof for a strictly \(k\)-Stein space \(X\) with \(W\)-exhaustion \(\{X_i\}_{i\in\bZ_{\geq0}}\), one needs to be able to construct a finite collection of global sections \(g_1,\ldots,g_n\in\mathcal{O}_X(X)\) such that
	\[\inf_{x\in X\setminus X_i}\sup\left\{|g_1(x)|,\ldots,|g_n(x)|\right\}\geq\alpha_i\qquad\text{for every}\quad i\in\bZ_{\geq0},\]
	for a certain sequence \((\alpha_i)_{i\in\bZ_{\geq0}}\) in \(\bR_{>0}\). To do this in the complex analytic setting, one makes explicit use of the underlying real analytic structure of \(X\) (cf. the proof of \cite[Theorem 1]{Narasimhan}). The natural question now arises whether an alternative proof is possible in the non-Archimedean setting.
\end{remark}
\section{Runge immersions}\label{secRunge}
In this section, we define Runge immersions into \(k\)-Stein spaces, and show that the analogues of several classical results extend to this setting. The latter will serve as auxiliary results when proving a Gerritzen-Grauert Theorem for \(k\)-Stein spaces in Section \ref{secGG}. We start by discussing natural generalisations of rational domains and affinoid generating systems.
\begin{remark}\label{rem:ratdomstein}
	Let \(\A=\varprojlim_i\A_i\) be a (strictly) \(k\)-Stein algebra. Then (strictly) Weierstrass, Laurent and rational domains of \(X=\spb(\A)\) are defined just as for Berkovich spectra of \(k\)-Banach algebras in Definition \ref{def:rationaldomain}. By \cite[Lemma 3.9]{MaculanPoineau}, (strictly) rational domains of \(X\) are also (strictly) \(k\)-Stein spaces. Furthermore, one easily verifies that these are \(G\)-open subsets and that the coordinate ring of a rational domain
	\[X\left(r_\bullet^{-1}\frac{f_\bullet}{g}\right)=\left\{x\in\spb(\A)\,\,\middle|\,\,|f_j(x)|\leq r_j|g(x)|\text{ for }j=1,\ldots,n\right\}\]
	is given by \(\A\langle r_\bullet^{-1}f_\bullet/g\rangle\vcentcolon=\varprojlim_i\A_i\langle r_\bullet^{-1}f_\bullet/g\rangle\), each element of which can be written as a series \(\sum_Ja_J(f_\bullet/g)^J\) with each \(a_J\in\A\) and \(\lim_{|J|\to+\infty}\norm{a_J}_{i}r_\bullet^J=0\) for every \(i\in\bZ_{\geq0}\). Here, \(\norm{\cdot}_i\) denotes the seminorm on \(\A\) induced by \(\A\to\A_i\) via pullback for \(i\in\bZ_{\geq0}\).
\end{remark}

\begin{remark}\label{rem:affgen}
	We slightly generalise the notion of affinoid generating systems from \cite[§7.2.5]{BGR}. Let \(\varphi:\A=\varprojlim_i\A_i\to\B\) be a morphism of \(k\)-Stein algebras with \(\B\) \(k\)-affinoid. Take \(i_0\in\bZ_{\geq0}\) such that \(\varphi\) factorises through \(\A_i\) for all integers \(i\geq i_0\). As \(\B\) is \(k\)-affinoid, we know that there exists an affinoid generating system of \(\B\) over \(k\), which is thus also an affinoid generating system of \(\B\) over \(\A_i\) for each integer \(i\geq i_0\). That is, there are elements \(b_1,\ldots,b_n\in\B\) and a polyradius \(r_\bullet=(r_1,\ldots,r_n)\in\bR_{>0}^n\) such that we have a commutative diagram
	\[\begin{tikzcd}
		& {k\langle r_\bullet^{-1}\zeta_\bullet\rangle} \arrow[rd, "\beta", two heads] \arrow[ld, "\alpha_i"'] &   \\
		{\A_i\langle r_\bullet^{-1}\zeta_\bullet\rangle} \arrow[rr, "\gamma_i", two heads] &                                                                                                    & \B
	\end{tikzcd}\]
	where \(\beta\) and \(\gamma_i\) are surjective, sending the variable \(\zeta_j\) to \(b_j\) for each \(j\in\{1,\ldots,n\}\). It follows that \(\beta\) factorises through \(\varprojlim_{i\geq i_0}\A_i\langle r_\bullet^{-1}\zeta_\bullet\rangle=\A\langle r_\bullet^{-1}\zeta_\bullet\rangle\), which also maps surjectively to \(\B\).
	
	We now define an \emph{affinoid generating system of \(\B\) over \(\A\)} as a collection of elements \(b_1,\ldots,b_n\in\B\) and a polyradius  \(r_\bullet=(r_1,\ldots,r_n)\in\bR_{>0}^n\) such that there exist a commutative diagram
	\[\begin{tikzcd}
		\A \arrow[rr, "\varphi"] \arrow[rd] &                                                                & \B \\
		& {\A\langle r_\bullet^{-1}\zeta_\bullet\rangle} \arrow[ru, "\gamma"', two heads] &  
	\end{tikzcd}\]
	where \(\gamma\) is surjective with \(\gamma(\zeta_j)=b_j\) for every \(j\in\{1,\ldots,n\}\). When \(k\) is non-trivially valued with \(\A\) strictly \(k\)-Stein and \(\B\) strictly \(k\)-affinoid, \(b_1,\ldots,b_n\in\B\) is called a \emph{strictly affinoid generating system} if above we can take \(r_1=\ldots=r_n=1\). Note that we have just proven that (strictly) affinoid generating systems always exists.
\end{remark}
\begin{definition}
	A morphism \(X\to Y\) of (strictly) \(k\)-Stein spaces with \(X\) \(k\)-affinoid is called a \emph{(strictly) Runge immersion} if it factorises as the composition of a closed immersion and a (strictly) Weierstrass domain embedding.
\end{definition}
\begin{proposition}\label{prop:equivrunge}
	Let \(\varphi:\A\to \B\) be a morphism of (strictly) \(k\)-Stein algebras with \(\B\) (strictly) \(k\)-affinoid. Then the following are equivalent.
	\begin{enumerate}[label=(\roman*)]
		\item \(\spb(\varphi)\) is a (strictly) Runge immersion.
		\item \(\varphi(\A)\) is dense in \(\B\).
		\item \(\varphi(\A)\) contains a (strictly) affinoid generating system of \(\B\) over \(\A\).
	\end{enumerate}
\end{proposition}
\begin{proof}
	One can follow the same argument as in \cite[Proposition 7.3.4/2]{BGR}. For implication \((ii)\implies(iii)\), one uses the lemma below, which is itself a generalisation of \cite[Lemma 7.3.4/3]{BGR}.
\end{proof}
\begin{lemma}
	Let \(\sigma:\A\to\B\) be a morphism of (strictly) \(k\)-Stein algebras with \(\B\) (strictly) \(k\)-affinoid and \(b_1,\ldots,b_n\in\B\) a (strictly) affinoid generating system of \(\B\) over \(\A\) with associated polyradius \(r_\bullet=(r_1,\ldots,r_n)\in\bR_{>0}^n\). Then there exists an \(\varepsilon\in\bR_{>0}\) such that the following holds: if \(d_1,\ldots,d_n\) is a system of elements in \(\B\) satisfying \(\norm{b_i-d_i}_\B<\varepsilon\) for every \(i\in\{1,\ldots,n\}\), then \(d_1,\ldots,d_n\) is a (strictly) affinoid generating system of \(\B\) over \(\A\) with associated polyradius \(r_\bullet\).
\end{lemma}
\begin{proof}
	By the obvious generalisation of \cite[Proposition 2.1.2]{Berkovich} for \(k\)-Fr\'echet algebras, we may assume that \(k\) is non-trivially valued. Fix a presentation \(\A=\varprojlim_i\A_i\) and let \(\sigma':\A\langle r_\bullet^{-1}\zeta_\bullet\rangle\to\B\) be the surjective morphism extending \(\sigma\) by mapping \(\zeta_j\) to \(b_j\) for each \(j\in\{1,\ldots,n\}\). Denote by \(\norm{\cdot}_{\A,i}\) the seminorm on \(\A\) induced by \(\A\to\A_i\) via pullback for each \(i\in\bZ_{\geq0}\). By Banach's Open Mapping Theorem, the norm \(\norm{\cdot}_\B\) and the collection of residue seminorms \(\{\norm{\cdot}_{\B,i}\}_{i\in\bZ_{\geq0}}\) on \(\B\) corresponding to \(\{\norm{\cdot}_{\A,i}\}_{i\in\bZ_{\geq0}}\) via \(\sigma'\), induce equivalent \(k\)-Fr\'echet structures on \(\B\). Since the family of seminorms \(\{\norm{\cdot}_{\B,i}\}_{i\in\bZ_{\geq0}}\) is increasing on \(\B\), we thus get for sufficiently large \(i\in\bZ_{\geq0}\) that \(\norm{\cdot}_{\B,i}\) is a norm on \(\B\) equivalent to \(\norm{\cdot}_\B\). Now one can give the same argument as in \cite[Lemma 7.3.4/3]{BGR}.
\end{proof}
\begin{proposition}\label{prop:ratindense}
	Let \(D\) be a (strictly) Liu domain in a (strictly) \(k\)-Stein space \(X\) such that \(\mathcal{O}_X(X)\) has dense image in \(\mathcal{O}_X(D)\). Any (strictly) rational (resp. Weierstrass) domain of \(D\), contained in the interior of \(D\), can be written as a finite union of connected components of a (strictly) rational (resp. Weierstrass) domain of \(X\).
\end{proposition}
\begin{proof}
	By the same argument as in the beginning of the proof of Proposition \ref{prop:basechangeliufull}, it follows from the assumption that \(\mathcal{O}_X(X)\) has dense image in \(\mathcal{O}_X(D)\) that we can take \(g'\) and \(f_\bullet'\) in \(\mathcal{O}_X(X)\) such that \(R=D(r_\bullet^{-1}f_\bullet'/g')\). Since \(R\) is compact and contained in the interior of \(D\), it is a finite union of connected components of \(X(r_\bullet^{-1}f_\bullet'/g')\).
\end{proof}
\begin{corollary}\label{cor:transcomp}
	Let \(X\) be a (strictly) \(k\)-Stein space. Let \(U\subset X\) be a (strictly) rational (resp. Weierstrass) domain and \(V\subset U\) a compact (strictly) rational (resp. Weierstrass) domain. Then \(V\) is a finite union of connected components of a (strictly) rational (resp. Weierstrass) domain of \(X\). In case \(U\) is also compact, \(V\) is a (strictly) rational (resp. Weierstrass) domain of \(X\).
\end{corollary}
\begin{proof}
	 When \(U\) is compact, one can copy the argument from \cite[Theorem 7.2.4/2]{BGR}. For the general case, let \(\{D_i\}_{i\in\bZ_{\geq0}}\) be a \(W\)-exhaustion of \(X\) by (strictly) Liu domains. As \(V\) is compact, it is contained in the interior of \(D_i\) for some \(i\in\bZ_{\geq0}\). It now follows from the compact case that \(V\) is a (strictly) rational (resp. Weierstrass) domain of \(D\). The claim now follows directly from Proposition \ref{prop:ratindense}.
\end{proof}
\begin{proposition}\label{prop:openclosedweierstrass}
	Let \(\varphi:X\to Y\) be a \(G\)-open and closed immersion of (strictly) \(k\)-Stein spaces with \(X\) (strictly) \(k\)-affinoid. Then \(X\) is a finite union of connected components of a (strictly) Weierstrass domain in \(Y\) via \(\varphi\).
\end{proposition}
\begin{proof}
	In case the proposition holds in the strict case, one can follow Ducros's argument for drawing the Gerritzen-Grauert Theorem for affinoid spaces from the strictly affinoid case \cite[Lemme 2.4]{Ducros}, to derive the general case by base change descent. Hence, assume that \(k\) is non-trivially valued with \(X\) and \(Y\) strictly \(k\)-Stein spaces. When \(Y\) is compact, one can use Corollary \ref{cor:noetherian} and \cite[Proposition 2.4]{LiuTohoku} to give the same argument as in \cite[Proposition 7.3.3/7]{BGR} and see that \(X\) is a strictly Weierstrass domain in \(Y\). When \(Y\) is not compact, let \(\{D_i\}_{i\in\bZ_{\geq0}}\) be a \(W\)-exhaustion of \(Y\) by strictly Liu domains. Since \(X\) is compact, it is contained in the interior of \(D_i\) for \(i\in\bZ_{\geq0}\) sufficiently large, and thus is a strictly Weierstrass domain in \(D_i\). One concludes by applying Proposition \ref{prop:ratindense}.
\end{proof}
\begin{corollary}\label{cor:openrunge}
	Let \(\varphi:X\hookrightarrow Y\) be a (strictly) \(k\)-affinoid domain embedding into a (strictly) \(k\)-Stein space which is also a (strictly) Runge immersion. Then \(X\) is a finite union of connected components of a (strictly) Weierstrass domain in \(Y\) via \(\varphi\).
\end{corollary}
\begin{proof}
	Using Proposition \ref{prop:openclosedweierstrass} and Corollary \ref{cor:transcomp}, this can be verified in the same way as in \cite[Proposition 7.3.4/6]{BGR}.
\end{proof}

\section{Gerritzen-Grauert Theorem for Stein spaces}\label{secGG}
In this section, we prove a generalisation of the Gerritzen-Grauert Theorem to \(k\)-Stein spaces. We start by extending the description of monomorphisms of strictly \(k\)-affinoid spaces as locally closed immersions to the strictly \(k\)-Stein setting.

\begin{proposition}\label{prop:monomorphism}
	Assume that \(k\) is non-trivially valued. A morphism \(\varphi:Y\to X\) of strictly \(k\)-Stein spaces is a locally closed immersion if and only if it is a monomorphism in the category of strictly \(k\)-Stein spaces.
\end{proposition}
\begin{proof}
	One can give the same argument as is done for the strictly affine case in \cite[Proposition 1.2]{TemkinGG}. In particular, for the converse implication, as being a locally closed immersion is a local property, one may assume that \(Y\) is strictly \(k\)-affinoid, so that one can use the lemma below.
\end{proof}
\begin{lemma}\label{lem:bgr7252}
	Assume that \(k\) is non-trivially valued. Let \(\sigma:\A\to\B\) be a morphism of strictly \(k\)-Stein algebras with \(\B\) strictly \(k\)-affinoid and let \(x\in\spr(\A)\). Assume that \(\sigma\) induces
	\begin{enumerate}[label=(\roman*)]
		\item a finite morphism \(\A/\mathfrak{m}_x\to\B/\mathfrak{m}_x\B\), or
		\item an epimorphism \(\A/\mathfrak{m}_x\to\B/\mathfrak{m}_x\B\), or
		\item isomorphisms \(\A/\mathfrak{m}_x^n\to\B/\mathfrak{m}_x^n\B\) for all \(n\in\bZ_{>0}\).
	\end{enumerate}
	Then there exists a strictly affinoid domain \(\spb(\A')\subset\spb(\A)\) containing \(x\) such that the map \(\sigma':\A'\to\A'\widehat{\otimes}_\A\B\) induced by \(\sigma\) is finite, epimorphic, or isomorphic, respectively.
\end{lemma}
\begin{proof}
	By Lemma \ref{lem:maxfingen}(ii), we may assume that \(\A\) is also strictly \(k\)-affinoid, in which case this is just \cite[Lemma 7.2.5/2]{BGR}.
\end{proof}
\begin{theorem}\label{thm:ggstein}
	Let \(\varphi:Y\to X\) be a monomorphism of (strictly) \(k\)-Stein spaces with \(Y\) (strictly) \(k\)-affinoid. Then there exists an admissible covering \(\{R_i\}_{i\in I}\) of \(X\) by (strictly) rational domains such that \(\varphi\) induces (strictly) Runge immersions \(\varphi_i:\varphi^{-1}(R_i)\to R_i\).
\end{theorem}
\begin{proof}	
	For convenience, we only consider the general case. The proof for the strict case is identical. First suppose that \(X\) is \(k\)-Liu and let \(X_1,\ldots,X_n\) be \(k\)-affinoid rational domains covering \(X\). By \cite[Theorem 3.1]{TemkinGG}, each \(X_i\) can be covered by rational domains \(R_{i1},\ldots,R_{in_i}\) such that each \(\varphi_{ij}:\varphi^{-1}(R_{ij})\to R_{ij}\) is a Runge immersion. By Corollary \ref{cor:transcomp}, each \(R_{ij}\) is also a rational domain of \(X\).
	
	Now suppose that \(X\) is a general \(k\)-Stein space and let \(\{D_i\}_{i\in\bZ_{\geq0}}\) be a \(W\)-exhaustion of \(X\) by Liu domains. Take \(j\in\bZ_{\geq0}\) sufficiently large such that the compact set \(\varphi(Y)\) is contained in the interior of \(D_j\). By the case above, there exists a covering \(\{R_1',\ldots,R_n'\}\) of \(D_j\) by rational domains such that \(\varphi^{-1}(R_i')\to R_i'\) is a Runge immersion for every \(i\in\{1,\ldots,n\}\). By Proposition \ref{prop:ratindense}, there exists a collection of rational domains \(\{R_1,\ldots,R_n\}\) of \(X\) such that \(R_i'\) is a finite union of connected components of \(R_i\) for every \(i\in\{1,\ldots,n\}\). In particular, each \(R_i'\) is a Weierstrass domain of \(R_i\), since we can take a section \(h\in\mathcal{O}_X(R_i)\) such that \(h=0\) on \(R_i\) and \(h=1\) on \(R_i'\). Hence, \(\varphi_i:\varphi^{-1}(R_i)\to R_i\) is a Runge immersion by Proposition \ref{prop:equivrunge}(ii). Lastly, since \(\varphi(Y)\) is compact, we can enlarge \(\{R_1,\ldots,R_n\}\) to a covering of \(X\) by repeating the process above outside of \(\varphi(Y)\), i.e. for every \(x\in X\setminus\varphi(Y)\) we can consider an affinoid neighbourhood \(U\) contained in the interior of \(D_\ell\setminus\varphi(Y)\) for some integer \(\ell\geq j\) and write \(U\) as the intersection of rational domains of \(X\) with \(D_\ell\). As the inverse image of such a latter rational domain \(R\) under \(\varphi\) is empty, the induced morphism \(\varphi^{-1}(R)=\varnothing\to R\) is obviously a Runge immersion.
\end{proof}

\begin{corollary}\label{cor:affinoidrational}
	Let \(X\) be a (strictly) \(k\)-Stein space. Then every (strictly) Liu domain in \(X\) is a finite union of (strictly) \(k\)-affinoid connected components of (strictly) rational domains of \(X\).
\end{corollary}
\begin{proof}
	Let \(Y\) be a (strictly) Liu domain in \(X\). By considering a finite covering of \(Y\) by (strictly) affinoid domains, the claim follows directly from Theorem \ref{thm:ggstein}, Corollary \ref{cor:openrunge} and Corollary \ref{cor:transcomp}. Alternatively, one can also repeat the argument of the proof of Theorem \ref{thm:ggstein}, using \cite[Proposition 3.12]{MaculanPoineau} instead of \cite[Theorem 3.1]{TemkinGG}.
\end{proof}

\end{document}